\documentclass{article}
\usepackage{amsmath,amssymb}
\usepackage{amsthm}
\usepackage{booktabs}
\usepackage{hyperref}
\usepackage{algorithm}
\usepackage{algpseudocode,comment}

\theoremstyle{plain}
\newtheorem{theorem}{Theorem}[section]
\newtheorem{lemma}[theorem]{Lemma}

\theoremstyle{definition}
\newtheorem{definition}[theorem]{Definition}
\newtheorem{example}[theorem]{Example}

\theoremstyle{remark}
\newtheorem{remark}[theorem]{Remark}

\usepackage{cleveref}

\usepackage{graphicx}
\usepackage{mathtools,xcolor}
\usepackage{caption}
\usepackage{newpxtext}
\usepackage{float}
\usepackage{tikz}
\usetikzlibrary{positioning,calc,decorations.pathreplacing}

\usepackage{pgfplots}
\pgfplotsset{compat=1.18} 

\usepackage{subcaption}
\DeclareMathOperator{\ord}{ord}
\DeclareMathOperator{\EE}{\mathbb{E}}

\DeclareMathOperator{\TV}{TV}
\DeclareMathOperator{\diam}{diam}
\DeclareMathOperator{\cyc}{cyc}
\newcommand{\rot}[1]{\phi_{#1}} 

\RequirePackage{parskip}
\RequirePackage{mathpazo}

\RequirePackage[backend=biber,backref=true,style=numeric,giveninits=true,natbib=true,maxbibnames=12]{biblatex}
\RequirePackage{microtype}

\newcommand{\Unif}{\mathrm{Unif}}
\newcommand{\calO}{\mathcal{O}}
\newcommand{\calI}{\mathcal{I}}
\newcommand{\ext}{\mathrm{ext}}
\newcommand{\new}{\mathrm{new}}

\newcommand{\cat}{{\mathrm{Cat}}}
\newcommand{\orbit}{\mathrm{orbit}} 
\newcommand{\ind}{\mathrm{index}}
\newcommand{\joint}{{\mathrm{joint}}}

\title{From Continuous to Discrete: \\ 
a No-U-Turn Sampler for Permutations}
\author{Nawaf Bou-Rabee\thanks{Department of Mathematical Sciences, Rutgers University, and Center for Computational Mathematics, Flatiron Institute \href{mailto:nawaf.bourabee@rutgers.edu}{\texttt{nawaf.bourabee@rutgers.edu}}}
\and
Zichu Wang\thanks{Courant Institute of Mathematical Sciences, New York University, \href{mailto:zw3409@nyu.edu}{\texttt{zw3409@nyu.edu}}}
}
\date{}

\begin{document}

\maketitle

\begin{abstract}
We introduce a discrete-space analogue of the No-U-Turn sampler on the symmetric group $S_n$, yielding a locally adaptive and reversible Markov chain Monte Carlo method for $\mathrm{Mallows}(d,\sigma_0)$.  Here $d:S_n\times S_n\to[0,\infty)$ is any fixed distance on $S_n$, $\sigma_0\in S_n$ is a fixed reference permutation, and the target distribution on $S_n$ has mass function $\pi(\sigma)\propto e^{-\beta d(\sigma,\sigma_0)}$ where $\beta>0$ is the inverse temperature.  The construction replaces Hamiltonian trajectories with measure-preserving group-orbit exploration. A
randomized dyadic expansion is used to explore a one-dimensional orbit until a
probabilistic \emph{no-underrun} criterion is met, after which the next state is
sampled from the explored orbit with probability proportional to the target
weights.  On the theory side, embedding this transition within the Gibbs self-tuning (GIST) framework provides a concise proof of reversibility. Moreover, we construct a \emph{shift coupling} for orbit segments and prove an explicit edge-wise contraction in the Cayley distance under a mild Lipschitz condition on the energy $E(\sigma)=d(\sigma,\sigma_0)$. A path-coupling argument then yields
an $O(n^2\log n)$ total-variation mixing-time bound.

\end{abstract}

\section{Introduction}

Locally adaptive Markov chain Monte Carlo (MCMC) methods such as the No-U-Turn Sampler (NUTS)~\cite{HoGe2014,betancourt2017conceptual,BouRabeeCarpenterMarsden2024} and its gradient-free analogue, the No-Underrun Sampler (NURS)~\cite{NURS1}, represent a major advance in MCMC.  Both
methods dynamically adapt the exploration length of proposed trajectories to
the local structure of the target distribution, enabling efficient mixing
without manual parameter tuning.  While NUTS and NURS have proved highly
effective for complex distributions in continuous spaces, no general framework
has yet extended these ideas to discrete settings such as permutation spaces,
where differentiable structure is absent and local scales vary combinatorially.

This paper introduces a discrete-space analogue of NUTS on the symmetric group $S_n$, yielding a locally adaptive and reversible sampler for the \emph{Mallows permutation model}. The Mallows model, denoted $\mathrm{Mallows}(d,\sigma_0)$, is a widely used distribution for ranking and preference data, as well as statistical inference on permutation spaces~\cite{vitelli2018mallows}. It forms a canonical exponential family on $S_n$ defined by \begin{equation}\label{eq:mallows} P_\beta(\sigma) =\frac{1}{Z_\beta}\exp[-\beta\, d(\sigma,\sigma_0)],\qquad \sigma\in S_n, \end{equation} where $d:S_n\times S_n\to\mathbb{R}_+$ is a permutation distance, $\sigma_0$ is a reference permutation, and $\beta>0$ controls the concentration of the distribution. The normalization constant (or partition function)
\[ Z_{\beta} = \sum_{\rho \in S_n} \exp\!\bigl(-\beta\, d(\rho, \sigma_0)\bigr) \;, \] ensures that $P_\beta$ defines a valid probability distribution. As
$\beta\to0$, the distribution converges to the uniform measure on $S_n$, while
as $\beta\to\infty$ it concentrates at $\sigma_0$.

Several standard choices of permutation distance arise in applications.  Let
$\sigma,\tau\in S_n$.  Common examples include:
\begin{itemize}
  \item \textbf{$L^1$ distance:}
  \(d_{L^1}(\sigma,\sigma_0) = \sum_{i=1}^n |\sigma(i) - \sigma_0(i)|\),
  measuring total absolute displacement.
  \item \textbf{$L^2$ distance:}
  \(d_{L^2}(\sigma,\sigma_0) = \sum_{i=1}^n (\sigma(i) - \sigma_0(i))^2\),
  penalizing large displacements quadratically.
  \item \textbf{Kendall distance:}
  \(d_\tau(\sigma,\sigma_0)\), the number of adjacent transpositions required to
  transform $\sigma^{-1}$ into $\sigma_0^{-1}$.
  \item \textbf{Cayley distance:}
  \(d_{\mathrm{Cay}}(\sigma,\sigma_0)\), the minimal number of transpositions
  needed to transform $\sigma$ into $\sigma_0$.
  \item \textbf{Hamming distance:}
  \(d_{\mathrm{Ham}}(\sigma,\sigma_0)
    = \#\{\, i : \sigma(i) \ne \sigma_0(i) \,\}\),
  counting the number of mismatched positions.
  \item \textbf{Ulam distance:}
  \(d_{\mathrm{Ulam}}(\sigma,\sigma_0)
    = n - \ell(\sigma_0\sigma^{-1})\),
  where $\ell(\sigma_0\sigma^{-1})$ denotes the length of the longest increasing
  subsequence of $\sigma_0\sigma^{-1}$.
\end{itemize}

Throughout the paper, we take the reference permutation to be
$\sigma_0=\mathrm{id}$ and analyze the resulting Markov kernel induced by the
discrete NURS transition from both algorithmic and probabilistic perspectives,
establishing reversibility and quantitative mixing-time bounds.

\subsection{Background and Motivation} 

We briefly review existing methods for sampling from the Mallows model and the
associated mixing-time analyses.  Classical Markov chain Monte Carlo (MCMC)
algorithms for sampling from~\eqref{eq:mallows} are based on local moves, such as
adjacent or random transpositions, combined with Metropolis--Hastings
accept--reject steps
\cite{DiaconisHanlon1992,DiaconisRam2000}.  In contrast, global-move algorithms
such as Hit-and-Run~\cite{Diaconis2007HitandRun} belong to a broader class of
auxiliary-variable MCMC methods that includes Gibbs sampling, data augmentation,
and slice sampling.  At each step, Hit-and-Run selects a random one-dimensional
direction (a ``line'') through the current state and samples the next point from
the conditional distribution along that line.

In discrete spaces such as $S_n$, the analogue of a line is a finite orbit under a
group action, and a Hit-and-Run update would require exact sampling from the
target distribution conditioned on such an orbit.  For a general permutation
distance $d$, this conditional distribution does not admit a tractable
representation: the order of a typical permutation grows superpolynomially in
$n$ (more precisely, $\log \operatorname{ord}(\pi) \sim \tfrac12(\log n)^2$ in
probability for $\pi$ uniform in $S_n$ by the Erd{\H{o}}s--Tur{\'a}n law
\cite{ErdosTuran1965,BarbourTavare1994}) and the energy values
$d(\sigma\circ\eta^k,\sigma_0)$ along the orbit do not in general exhibit exploitable
structure as a function of $k$.  As a consequence, exact conditional sampling
along orbits is computationally infeasible in general.  Global Hit-and-Run–type
updates are therefore available only in special cases where the energy admits a
separable form that enables efficient conditional sampling, as in the $L^1$ and
$L^2$ Mallows models studied by Zhong~\cite{zhong2021}.

\paragraph{Path coupling and mixing-time analyses.} Some quantitative convergence results are obtained via coupling arguments, and
in particular through the \emph{path coupling theorem}
\cite{BubleyDyer1997,LevinPeres2017}.  Path coupling reduces the task of
establishing contraction of the Markov kernel in expected distance over the
entire state space to verifying contraction for pairs of states at distance one
with respect to a chosen path metric.  For Markov chains on $S_n$, the state
space is commonly equipped with the path metric induced either by adjacent
transpositions (equivalently, Kendall’s $\tau$ distance) or by arbitrary
transpositions (equivalently, the Cayley distance).  One then constructs a
one-step coupling for pairs of permutations differing by a single transposition
and verifies that the expected distance after one step is strictly smaller than
the initial distance.

When such an edgewise contraction holds uniformly, the path coupling theorem
(see Theorem~\ref{thm:path-coupling}) implies exponential convergence to
stationarity in total variation.  The resulting mixing-time bound depends on the
contraction coefficient and on the diameter of the underlying metric space.
Alternative approaches to mixing-time analysis on $S_n$ include spectral
techniques, comparison arguments, and representation-theoretic methods, which
exploit the group structure of $S_n$ to analyze the spectrum of the transition
operator or to relate the chain to better-understood shuffling processes
\cite{DiaconisShahshahani1981,DiaconisHanlon1992,DiaconisSaloffCoste1993,DiaconisSaloffCoste1998,DiaconisRam2000}.

\paragraph{Card-Shuffling}
For $\beta=0$, the Mallows distribution reduces to the uniform measure on $S_n$,
independently of the choice of distance $d$ and reference permutation
$\sigma_0$.  In this case, there is no acceptance–rejection bias, and the
resulting Markov chain coincides with the underlying proposal chain.  Its
mixing behavior is therefore governed entirely by the proposal mechanism.

\begin{itemize}
  \item \textbf{Random transpositions.}
  In the random–transposition shuffle, at each step we choose
  $L,R \in \{1,\dots,n\}$ independently and uniformly and update
  $\sigma' = \sigma \circ (L\,R)$, with the convention that no move is made when
  $L=R$.  The resulting chain is irreducible and reversible with respect to
  $\Unif(S_n)$, and its total-variation mixing time is of order $O(n\log n)$; see
  \cite[Theorem~1]{DiaconisShahshahani1981}.

  \item \textbf{Adjacent transpositions.}
In the adjacent–transposition shuffle, at each discrete step we choose
$i \in \{1,\dots,n-1\}$ uniformly and update
$\sigma' = \sigma \circ (i\,i{+}1)$.  This chain is irreducible and reversible
with respect to $\Unif(S_n)$.  Its discrete-time total-variation mixing time is
$O(n^{3}\log n)$; this follows from the $O(n^{2}\log n)$ mixing time of the
continuous-time chain together with the fact that the continuous-time dynamics
performs $\Theta(n)$ adjacent transpositions per unit time, whereas the
discrete-time chain performs exactly one per step
\cite[Theorems~B.2 and~2.2]{Lacoin2016}.
\end{itemize}

\paragraph{Metropolis algorithm for the Cayley--Mallows model.}
Let $\theta = e^{-\beta}\in(0,1)$.  The Metropolis--Hastings algorithm on $S_n$
targeting the Cayley--Mallows distribution
$\mathrm{Mallows}(d_{\mathrm{Cay}},\sigma_0)$, as studied in
\cite{DiaconisSaloffCoste1998}, uses the random–transposition walk as its proposal
kernel.  From the current state $\sigma$, one chooses $i,j\in\{1,\dots,n\}$
uniformly at random and proposes $\sigma'=\sigma\circ(i\,j)$, with the proposal
accepted with probability
\[
\min\!\bigl\{1,\;\theta^{\,d_{\mathrm{Cay}}(\sigma',\sigma_0)
- d_{\mathrm{Cay}}(\sigma,\sigma_0)}\bigr\}.
\]
The resulting Markov chain is irreducible and reversible with respect to the
Cayley--Mallows distribution. Moreover, for fixed $\theta\in(0,1)$ and the chain
started at the identity, \cite[Theorem~2.1]{DiaconisSaloffCoste1998} yields a
total-variation mixing time of order $O(n\log n)$.

\paragraph{Systematic scan for the Kendall--Mallows model.}
Let $\theta=e^{-\beta}\in(0,1)$.  Consider the Metropolis algorithm on $S_n$
targeting the Kendall--Mallows distribution
$\mathrm{Mallows}(d_{\tau},\sigma_0)$, using a systematic scan of adjacent
transpositions as the proposal mechanism.  Writing $s_i=(i\,i{+}1)$, one scan
consists of sequentially applying the Metropolis kernels associated with the
adjacent transpositions in the fixed order
$s_1,\dots,s_{n-1},s_{n-1},\dots,s_1$.  The resulting Markov chain is irreducible
and reversible with respect to the Kendall--Mallows distribution.  Moreover,
\cite[Theorem~1.4]{DiaconisRam2000} shows that for each fixed
$\theta\in(0,1)$, the chain mixes in $O(n)$ scans, corresponding to a total
variation mixing time of order $O(n^2)$ in terms of individual Metropolis
updates.

\paragraph{Hit-and-Run for Separable Distances.}
For Mallows models on $S_n$ based on the $L^1$ and $L^2$ distances,
Zhong~\cite{zhong2021} constructs global ``hit-and-run'' Markov chains by introducing auxiliary variables and then
sampling \emph{exactly} from the conditional distribution on a tractable combinatorial
constraint set.   Each update introduces auxiliary variables that induce
coordinatewise, one-sided constraints on the next permutation, and then samples
the next state uniformly from the set of permutations satisfying those
constraints.  The resulting Markov chains are reversible and differ
substantially from the local Metropolis chains traditionally used for sampling
from Mallows models.

\smallskip
\noindent\emph{The $L^2$ model.}
When $\sigma_0=\mathrm{id}$, the squared $L^2$ energy satisfies the identity
\[
\sum_{i=1}^n (\sigma(i)-i)^2
\;=\;
2\sum_{i=1}^n i^2 - 2\sum_{i=1}^n i\,\sigma(i),
\]
so, up to an additive constant, the target distribution can be written as
\[
\widetilde P_\beta(\sigma)
\;\propto\;
\exp\!\Bigl(2\beta\sum_{i=1}^n i\,\sigma(i)\Bigr),
\qquad \sigma\in S_n.
\]
This representation linearizes the dependence of the energy on each coordinate
when viewed conditionally.  One step of Zhong’s $L^2$ hit-and-run chain from
$\sigma$ proceeds as follows \cite[Section~2.1]{zhong2021}:
\begin{itemize}
\item[(i)] independently draw
$u_i\sim\Unif\!\bigl([0,e^{2\beta i\sigma(i)}]\bigr)$ for $i=1,\dots,n$ and set
$b_i=(\log u_i)/(2\beta i)$;
\item[(ii)] sample $\sigma'$ uniformly from the constrained set
$\{\tau\in S_n:\ \tau(i)\ge b_i\ \text{for all }i\}$.
\end{itemize}
The restricted-uniform draw in step~(ii) can be implemented efficiently by a
sequential placement procedure \cite[Section~2.1]{zhong2021}.

\smallskip
\noindent\emph{The $L^1$ model.}
An analogous construction holds for the $L^1$ distance.  In this case, the
absolute-deviation energy admits a decomposition into one-sided contributions,
leading to upper-bound constraints of the form $\tau(i)\le b_i$ and a
corresponding restricted-uniform update
\cite[Section~2.2]{zhong2021}.

\smallskip
\noindent\emph{Scope and limitations.}
These hit-and-run samplers rely critically on special algebraic identities that
render the $L^1$ and $L^2$ energies separable and permit efficient sampling from
the resulting constrained permutation sets.  For Mallows models based on more
general permutation distances or other $L^p$ costs with $p\neq1,2$, no
comparable separable representation is known in general, and the corresponding
conditional updates appear intractable.

\smallskip
\noindent\emph{Mixing times.}
For the $L^2$ hit-and-run chain, Zhong proves an $O(\log n)$ total-variation
mixing-time upper bound in the high-temperature regime $\beta=O(n^{-2})$ using a
path-coupling argument \cite[Theorem~3.2]{zhong2021} (more precisely, under the
condition $16\beta n^2<1$).  In the same scaling, the paper establishes an
$\Omega(n\log n)$ lower bound for a local Metropolis chain based on random
transpositions targeting the $L^2$ Mallows model
\cite[Theorem~3.6]{zhong2021}.

For the $L^1$ hit-and-run chain, an analogous $O(\log n)$ mixing-time upper bound
is proved under the weaker high-temperature condition $\beta=O(n^{-1})$
\cite[Theorem~3.4]{zhong2021}.  At the same scaling, the corresponding local
Metropolis chain is shown to have $\Omega(n\log n)$ mixing time
\cite[Theorem~3.7]{zhong2021}.  Together, these results highlight the substantial
efficiency gains achievable with global hit-and-run updates whenever such
constructions are available.

\paragraph{Local adaptivity in continuous spaces: a bridge to permutations.}
In continuous space settings, the \emph{No-U-Turn Sampler} (NUTS), a widely used variant
of Hamiltonian Monte Carlo, adapts the integration time by progressively
expanding a Hamiltonian trajectory in randomly chosen forward and backward
directions until a stopping criterion indicates that further exploration is
unlikely to improve sampling efficiency.  The classical \emph{no-U-turn} rule
halts expansion when the simulated trajectory begins to reverse direction in
position space.  Owing to this automatic tuning of the trajectory length, NUTS
has become the default sampler in many probabilistic programming environments
\cite{carpenter2016stan,salvatier2016probabilistic,nimble-article:2017,ge2018t,phan2019composable}.

Its gradient-free analogue, the No-Underrun Sampler (NURS), replaces the geometric
no-U-turn criterion with a probabilistic \emph{no-underrun} rule that terminates
trajectory expansion once the additional probability mass contributed by
further extension is negligible \cite{NURS1}.  The resulting transition samples
from all points visited along the constructed trajectory with probabilities
proportional to the target density, yielding a rejection-free and reversible
Markov kernel.  The key insight underlying the present work is that this
mechanism extends naturally to discrete spaces once continuous volume
preservation along Hamiltonian flows is replaced by measure preservation under
appropriate group actions.

\subsection{Main Results} 

We introduce and analyze a discrete analogue of the No-Underrun Sampler (NURS) on
the symmetric group $S_n$.  Given a current state $\sigma\in S_n$, a
\emph{direction} $\rho\in S_n$ is drawn from a (possibly state-dependent)
direction law $q(\cdot\mid\sigma)$, and the algorithm considers the
measure-preserving maps
\[
\sigma \longmapsto \sigma\circ\rho^{k}, \qquad k\in\mathbb{Z}.
\]
These maps generate a one-dimensional group orbit
\[
\mathcal{O}(\sigma,\rho)
\;:=\;
\{\sigma\circ\rho^{k} : k\in\mathbb{Z}\},
\]
whose cardinality equals the order $\ord(\rho)$ of the permutation $\rho$.

Starting from $\sigma$, the algorithm explores $\mathcal{O}(\sigma,\rho)$ using a
randomized dyadic-doubling procedure that recursively enlarges a contiguous
segment of the orbit.  Exploration terminates when a probabilistic
\emph{no-underrun} stopping rule determines that the cumulative target mass of the
 explored segment dominates the contribution of its boundary points.
Conditioned on the realized orbit segment $\mathcal{O}$, the next state
$\sigma'$ is sampled from the categorical distribution
\[
\mathbb{P}(\sigma'=\tau \mid \mathcal{O})
\;=\;
\frac{w(\tau)}{\sum_{\tau'\in\mathcal{O}} w(\tau')},
\qquad
w(\tau)=\exp\!\bigl(-\beta\, d(\tau,\sigma_0)\bigr),
\]
where $\beta>0$ is an inverse-temperature parameter and
$d:S_n\times S_n\to[0,\infty)$ is a fixed permutation distance.

\medskip
The algorithm depends on a user-specified \emph{threshold parameter}
$\varepsilon>0$, which enters the no-underrun stopping rule and controls the
randomized truncation of orbit exploration.  Smaller values of $\varepsilon$
lead to longer explored segments and higher computational cost per iteration,
while larger values shorten exploration.  The invariant distribution of the
resulting Markov chain does not depend on $\varepsilon$.

\medskip
Our analysis relies on a structural symmetry assumption on the direction law.
Specifically, we assume that $q(\cdot\mid\sigma)$ is \emph{orbit-equivariant}, in
the sense that its conditional distribution is invariant along each orbit
generated by any direction in its support:
\[
q(\,\cdot \mid \sigma\,)
\;=\;
q(\,\cdot \mid \sigma \circ \eta^{\,i}\,)
\qquad
\text{for all }\sigma\in S_n,\ \eta\in\operatorname{supp}(q),\text{ and integers } i.
\]
Equivalently, $q$ assigns the same distribution of directions to all states lying
on a common orbit. This equivariance property plays a central role in both of our
main results.

\medskip
\noindent\textbf{Reversibility.}
Our first main result, Theorem~\ref{thm:main-reversibility}, establishes that the
discrete No-Underrun Sampler (NURS) is reversible with respect to the Mallows target
distribution. The proof proceeds by expressing NURS as a special case of the Gibbs
Self-Tuning (GIST) framework, in which reversibility is obtained by augmenting the
state space with auxiliary variables and applying a carefully chosen involutive
map. In the NURS setting, the augmented state includes the sampled direction, the
randomized dyadic orbit generated by the doubling scheme, and the index selected
along that orbit.

The key structural step is the identification of an involution on this augmented
space that re-centers the orbit at the selected index, effectively exchanging the
roles of the current state and the proposed state along the same orbit. Reversibility
then follows by verifying that the extended target distribution—combining the
Mallows law, the orbit-equivariant direction distribution, the orbit selection
kernel induced by randomized doubling, and the index selection kernel defined by
Mallows weights—is invariant under this involution. This invariance renders the
Metropolis–Hastings correction trivial, so that NURS is rejection-free while still
satisfying detailed balance. Theorem~\ref{thm:main-reversibility} thus provides a
rigorous correctness guarantee for the discrete NURS transition.

\medskip
\noindent\textbf{Mixing time for shiftable directions.}
Our second main result, Theorem~\ref{thm:main-mixing}, provides a quantitative
mixing-time bound for an idealized version of the discrete NURS transition under a
concrete, structurally constrained family of direction laws. We work on the
transposition Cayley graph of $S_n$ equipped with the Cayley distance
$d_{\mathrm{Cay}}$, so that adjacent states are pairs
$(\sigma,\sigma\circ\tau_{ij})$ differing by a single transposition. The proof is
based on path coupling: we construct, for each edge, a one-step coupling whose
expected Cayley distance contracts by a uniform factor, and then propagate this
local contraction to all of $S_n$ via the Bubley–Dyer path coupling theorem.

The direction law is supported on \emph{shiftable} permutations of the form
$\eta=\tau_{ij}h$, where $\tau_{ij}$ is a transposition and $h$ fixes $i$ and $j$
and consists only of odd-length cycles (Definition~\ref{def:Omegai}). This algebraic
structure implies $\eta^{m}=\tau_{ij}$ for $m=\ord(h)$ and hence forces a rigid
relationship between $\eta$-orbits based at neighboring states. In the mixing-time
analysis we consider an idealized, non-truncated kernel in which the orbit window
is the full $\eta$-orbit (length $2m$) and the stopping and sub-stopping rules are
suppressed; this idealized kernel coincides with NURS when the orbit window is taken
to be of full length.

Fix an edge $(\sigma,\sigma\circ\tau_{ij})$ and couple the two chains by sharing the
random choice of the transposition pair $(I,J)$ and the subsequent direction
$\eta\in\Omega_{IJ}$. This produces two regimes. In the \emph{aligned} case
$(I,J)=(i,j)$, the orbit identity $\eta^{m}=\tau_{ij}$ yields an exact index-shift
relation $Y_t=X_{t+m}$ between the two full orbits; consequently the induced
categorical laws on the orbit window coincide up to a deterministic shift, and one
obtains a perfect one-step coupling with $d_{\mathrm{Cay}}(U,V)=0$. In the
complementary \emph{mismatch} case $(I,J)\neq(i,j)$, the two orbit points remain at
unit Cayley distance when indexed in lockstep, but the orbit segments do not align.
Here the coupling reduces to controlling the total variation distance between the two
categorical distributions obtained by normalizing Mallows weights along paired orbit
segments. Under the Lipschitz regularity assumption
$|E(\pi)-E(\rho)|\le L_E,d_{\mathrm{Cay}}(\pi,\rho)$, this weight control yields the
bound $\TV(p,q)\le \tanh(\beta L_E)$, enabling a maximal coupling that mismatches the
selected indices with probability at most $\tanh(\beta L_E)$.

Combining the aligned and mismatch contributions gives an explicit edge-wise
contraction factor $\delta(\beta)$ for the expected Cayley distance, and hence (when
$\delta(\beta)<1$) geometric decay of total variation distance. In the resulting
high-temperature regime, this yields an $O\bigl(n^2\log n\bigr)$ upper bound on the
total-variation mixing time.

\subsection{Structure of the Paper} Section~\ref{sec:nurs-discrete} introduces the discrete NURS transition.
Section~\ref{sec:nurs-reversibility} establishes reversibility via a Gibbs
self--tuning (GIST) embedding. Section~\ref{sec:coupling} develops the shift
coupling and derives a contraction bound.

Section~\ref{sec:numerics} reports numerical experiments for Mallows models based
on the Kendall, \(L^1\), \(L^2\), Hamming, Cayley, and Ulam distances, across a
range of inverse--temperature regimes. The experiments examine orbit index
distributions, fixed--point statistics, and trace diagnostics, and compare NURS
with standard two--point Metropolis--type updates. The results are consistent with
the qualitative and scaling predictions of the theory.

In Section~\ref{sec:beyond}, we conclude with open problems and directions for extending orbit--based sampling
beyond the symmetric group \(S_n\).

\section{The Discrete No-Underrun Sampler}
\label{sec:nurs-discrete}
This section formalizes the discrete No-Underrun Sampler (NURS) on the symmetric group~$S_n$
for target distributions of the form $P_\beta(\sigma)\propto \exp[-\beta\, d(\sigma,\sigma_0)]$
corresponding to the $\mathrm{Mallows}(d,\sigma_0)$ model in~\eqref{eq:mallows} with $\sigma_0$ set to the identity permutation.
Throughout, $\beta>0$ denotes the inverse-temperature parameter and $d:S_n\times S_n\to\mathbb{R}_+$ a chosen permutation distance.
For convenience, write
\begin{equation} \label{eq:weights}
w(\sigma)=\exp[-\beta\, d(\sigma,\sigma_0)] \;.
\end{equation}

\paragraph{Direction distribution.}
At each transition, the discrete NURS algorithm samples a random \emph{direction} on the symmetric group $S_n$.  
Formally, this direction is drawn from a probability distribution $q(\cdot \mid \sigma)$ that may depend on the current state $\sigma \in S_n$ and determines the orbit along which the next state of the chain is selected.  
To ensure consistency of orbit construction, the distribution $q$ must be invariant along each orbit it generates.

\medskip

\begin{definition}[Direction law]\label{def:direction-law}
 A probability kernel $q$ on $S_n$ is called an \emph{orbit-equivariant direction law} if
\[
q(\,\cdot \mid \sigma\,) = q(\,\cdot \mid \sigma \circ \eta^{\,i}\,)
\quad
\text{for all } \sigma \in S_n,\ \eta \in \operatorname{supp}(q(\cdot\mid\sigma)),\text{ and integers } i.
\]
\end{definition}

The orbit-equivariance property ensures that the direction distribution is identical for all states lying on the same orbit generated by any $\eta \in \operatorname{supp}(q(\cdot\mid\sigma))$.  
In the special case where $q$ is independent of~$\sigma$, the condition is automatically satisfied; the uniform distribution on~$S_n$ is the canonical example.

\begin{algorithm}[tbhp]
\caption{Fisher--Yates Shuffle for Sampling from $q=\Unif(\boldsymbol{S_n})$}
\label{alg:fisher-yates}
\textbf{Input:} positive integer $n$

\textbf{Output:} uniformly random permutation $\pi \in S_n$

\hrule

\begin{algorithmic}[1]
  \State Initialize $\pi = (\pi(1), \pi(2), \dots, \pi(n))$ by setting $\pi(i) = i$ for $i = 1, 2, \dots, n$
  \For{$i = n$ down to $2$}
    \State Choose $j$ uniformly at random from $\{1, 2, \dots, i\}$
    \State Swap $\pi(i)$ and $\pi(j)$
  \EndFor
  \State \Return $\pi$
\end{algorithmic}
\end{algorithm}

\medskip

\paragraph{Orbit Elements.} 

Orbits play a central role as auxiliary variables in the NURS transition.   We formalize below the notion of an orbit generated from an initial state and introduce a concatenation operation that will be used in the recursive construction of orbits.

\medskip

\begin{definition}[Orbit] \label{defn:orbit}
Let $\rho \in S_n$ be a fixed direction and $\sigma \in S_n$ an initial state. An \emph{orbit} generated by $\rho$ from $\sigma$ over the index interval $[a{:}b] = \{a, a+1, \dots, b\} \subset \mathbb{Z}$, with $a \le b$, is the ordered collection
\[
\mathcal{O} = (\sigma_a, \sigma_{a+1}, \dots, \sigma_b),
\]
where each element is defined by
\[
\sigma_k = \sigma \circ \rho^k.
\]
Here, $\rho^k$ denotes the $k$-fold composition of $\rho$ with itself if $k \ge 0$, and the $|k|$-fold composition of $\rho^{-1}$ if $k < 0$. The \emph{length} of the orbit, denoted $|\mathcal{O}|$, is $b - a + 1$.
\end{definition}

\medskip

\begin{definition}[Orbit Concatenation] \label{defn:concatenation}
Let $a, b, c \in \mathbb{Z}$ with $a \le b < c$, and fix $\rho \in S_n$ and $\sigma \in S_n$.  
Define two consecutive orbits generated by $\rho$ from $\sigma$:
\[
\calO = (\sigma_a, \sigma_{a+1}, \dots, \sigma_b) \quad \text{and} \quad \widetilde{\calO} = (\sigma_{b+1}, \dots, \sigma_c),
\]
where $\sigma_k = \sigma \circ \rho^k$ for $k \in [a:c]$.   The \emph{concatenation} of these orbits, denoted $\calO \odot \widetilde{\calO}$, is the combined sequence
\[
\calO \odot \widetilde{\calO} = (\sigma_a, \sigma_{a+1}, \dots, \sigma_c),
\]
with length $|\calO \odot \widetilde{\calO}| = c - a + 1$.
\end{definition}

\medskip
\paragraph{Orbit Construction.} 
Given $\rho \sim q(\cdot \mid \sigma)$, we construct an orbit starting from an initial permutation $\sigma \in S_n$ using a stochastic doubling procedure with a corresponding termination rule.

At each doubling step, we randomly choose a direction, either forward or backward, with equal probability. Suppose the current orbit $\mathcal{O}$ consists of permutations indexed from $a$ to $b$, and let $\ell = 2^{j-1}$ denote the length of the candidate extension at the $j$-th doubling.

If the direction is forward, we append the extension
\[
\mathcal{O}^{\ext} = (\sigma \circ \rho^{b+1},\ \sigma \circ \rho^{b+2},\ \dots,\ \sigma \circ \rho^{b + \ell}),
\]
and define the proposed updated orbit as $\mathcal{O}^{\new} = \mathcal{O} \odot \mathcal{O}^{\ext}$. If the direction is backward, we prepend the extension
\[
\mathcal{O}^{\ext} = (\sigma \circ \rho^{a - \ell},\ \sigma \circ \rho^{a - \ell + 1},\ \dots,\ \sigma \circ \rho^{a - 1}),
\]
and define $\mathcal{O}^{\new} = \mathcal{O}^{\ext} \odot \mathcal{O}$.

This process continues until one of the following criteria is met:
\begin{itemize}
  \item \textbf{Sub-stopping:} If the extension $\mathcal{O}^{\ext}$ satisfies a sub-stopping condition, we discard it and terminate with the current orbit $\mathcal{O}$ as the final orbit.
  \item \textbf{Stopping:} If the full extended orbit $\mathcal{O}^{\new}$ satisfies the stopping condition, we accept it and terminate with $\mathcal{O}^{\new}$ as the final orbit.
  \item \textbf{Maximum length:} If the extended orbit $\mathcal{O}^{\new}$ reaches the maximum permitted length $2^M$, we terminate and accept $\mathcal{O}^{\new}$ as the final orbit.
\end{itemize}
These stopping rules are based on a density threshold parameter $\varepsilon$ and are described next, followed by a detailed account of the state selection mechanism.

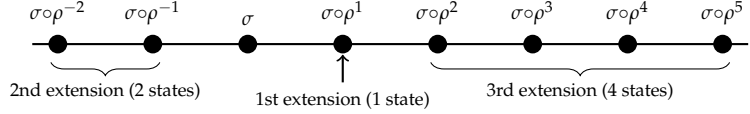
\begin{figure}[ht]
\centering
\begin{tikzpicture}[xscale=0.85, node distance=1cm]
  \tikzset{
    dot/.style={circle, fill=black, inner sep=2.5pt},
  }

  \node[dot,label=above:{$\scriptstyle\sigma \circ \rho^{-2}$}] (n-2) {};
  \node[dot,label=above:{$\scriptstyle\sigma \circ \rho^{-1}$}] (n-1) [right=of n-2] {};
  \node[dot,label=above:{$\scriptstyle\sigma$}] (n0) [right=of n-1] {};
  \node[dot,label=above:{$\scriptstyle\sigma \circ \rho^{1}$}] (n1) [right=of n0] {};
  \node[dot,label=above:{$\scriptstyle\sigma \circ \rho^{2}$}] (n2) [right=of n1] {};
  \node[dot,label=above:{$\scriptstyle\sigma \circ \rho^{3}$}] (n3) [right=of n2] {};
  \node[dot,label=above:{$\scriptstyle\sigma \circ \rho^{4}$}] (n4) [right=of n3] {};
  \node[dot,label=above:{$\scriptstyle\sigma \circ \rho^{5}$}] (n5) [right=of n4] {};

  \draw[thick] ([xshift=-4mm]n-2.center) -- ([xshift=4mm]n5.center);

  \node[text=black, font=\scriptsize] (firstLabel) at ($(n1)+(0,-0.75)$) {1st extension (1 state)};
  \draw[->, thick, black] (firstLabel.north) -- ($(n1)+(0,-4pt)$);


\draw [decorate,decoration={brace,mirror,amplitude=5pt}]
  ($(n-2.south west)+(0,-0.15)$) -- ($(n-1.south east)+(0,-0.15)$)
  node[midway, below=4pt] {\scriptsize 2nd extension (2 states)};

\draw [decorate,decoration={brace,mirror,amplitude=5pt}]
  ($(n2.south west)+(0,-0.15)$) -- ($(n5.south east)+(0,-0.15)$)
  node[midway, below=4pt] {\scriptsize 3rd extension (4 states)};

\end{tikzpicture}
\caption{\textbf{One realization of orbit construction in NURS by random doubling.} Starting from $\sigma$, the orbit expands forward, then backward, then forward again. 
The construction halts once either a \textsc{Stop} or \textsc{SubStop} condition is met, or when the maximum allowed trajectory length $2^M$ is reached, yielding the final orbit $\mathcal{O}$.}
\label{fig:orbit-example}
\end{figure}

\paragraph{Stopping and Sub-Stopping Conditions.}

Recall from~\eqref{eq:weights} that the unnormalized weights are given by 
$w(\sigma) = \exp[-\beta\, d(\sigma,\sigma_0)]$,
so that $P_\beta(\sigma) = w(\sigma) / \sum_{\rho \in S_n} w(\rho)$.
These weights underlie the following probabilistic stopping rules for orbit expansion.

\medskip

\begin{definition}[\textsc{Stop} (no--underrun condition)]
\label{defn:no-underrun}
Let $\mathcal{O} = (\sigma \circ \rho^k)_{k = a}^{b}$ be an orbit indexed by
consecutive integers from $a$ to $b$, and let $\varepsilon \in (0,1)$.
We say that $\mathcal{O}$ satisfies the \emph{no--underrun condition},
denoted $\textsc{Stop}(\mathcal{O},\varepsilon,w)$, if
\[
\max \!\left\{ w(\sigma \circ \rho^{a}),\, w(\sigma \circ \rho^{b}) \right\}
\;\le\;
\varepsilon \sum_{\tau \in \mathcal{O}} w(\tau).
\]
Equivalently, the combined weight at the two boundary points of the orbit
accounts for at most an $\varepsilon$ fraction of the total weight accumulated
along $\mathcal{O}$.
\end{definition}

\medskip

To ensure reversibility, we also introduce a recursive stopping rule that
prevents termination on shorter dyadic sub-orbits.

\medskip

\begin{definition}[Dyadic sub-orbits]
\label{def:dyadic-suborbit}
Let $\mathcal{O} = (\sigma \circ \rho^k)_{k=a}^{b}$ be an orbit of length
$\ell = 2^j$ for some $j \in \mathbb{N}$.
For each level $s \in \{0,\dots,j\}$, partition $\mathcal{O}$ into
$2^s$ consecutive sub-orbits of equal length $2^{j-s}$:
\[
\mathcal{O}
=
\mathcal{O}_{s,1} \odot \cdots \odot \mathcal{O}_{s,2^s},
\]
where $\mathcal{O}_{s,t}$ denotes the $t$-th sub-orbit at level $s$.  The sub-orbits $\{\mathcal{O}_{s,t}\}$ are called the \emph{dyadic sub-orbits}
of $\mathcal{O}$.  They correspond to the canonical dyadic (binary-tree)
decomposition induced by the doubling construction.
\end{definition}

\medskip

\begin{definition}[\textsc{SubStop} condition]
\label{defn:sub-stop}
Let $\mathcal{O}$ be an orbit of length $2^j$, and let
$\{\mathcal{O}_{s,t}\}$ denote its dyadic sub-orbits
(Definition~\ref{def:dyadic-suborbit}). We say that $\mathcal{O}$ satisfies the \emph{sub-stopping condition}
(\textsc{SubStop}) if at least one dyadic sub-orbit satisfies the
\textsc{Stop} condition:
\[
\textsc{SubStop}(\mathcal{O},\varepsilon,w)
\quad\Longleftrightarrow\quad
\exists\, s\in\{0,\dots,j\},\; t\in\{1,\dots,2^s\}
\text{ such that }
\textsc{Stop}\bigl(\mathcal{O}_{s,t},\varepsilon,w\bigr).
\]
\end{definition}

\begin{example}[Dyadic decomposition of an orbit]
Let $\mathcal{O} = (\sigma \circ \rho^k)_{k=-2}^{5}$ be an orbit of length $|\mathcal{O}| = 8 = 2^3$.  
At each level $s \in \{0,1,2,3\}$, the orbit is partitioned into $2^s$ consecutive sub-orbits 
$\mathcal{O}_{s,t}$ of equal length $2^{3-s}$, obtained by iteratively halving the parent orbit:
\begin{align*}
\mathcal{O}
&= \mathcal{O}_{0,1} = \mathcal{O}_{1,1} \odot \mathcal{O}_{1,2} = (\mathcal{O}_{2,1} \odot \mathcal{O}_{2,2}) \odot (\mathcal{O}_{2,3} \odot \mathcal{O}_{2,4}) \\
&= \bigl((\sigma\circ\rho^{-2} \odot \sigma\circ\rho^{-1}) \odot (\sigma \odot \sigma\circ\rho^{1})\bigr)
    \odot \bigl((\sigma\circ\rho^{2} \odot \sigma\circ\rho^{3}) \odot (\sigma\circ\rho^{4} \odot \sigma\circ\rho^{5})\bigr)\;.
\end{align*}
The tree below illustrates this dyadic structure.  
Each node $\mathcal{O}_{s,t}$ represents a consecutive sub-orbit of length $2^{3-s}$, 
and the leaves correspond to the individual orbit elements $\sigma \circ \rho^{-2}, \ldots, \sigma \circ \rho^{5}$.  
The \textsc{SubStop} condition holds as soon as any sub-orbit $\mathcal{O}_{s,t}$ satisfies the no-underrun rule.
\begin{center}
\begin{tikzpicture}[grow=down,level distance=1.2cm,
  every node/.style={font=\normalsize,align=center},
  edge from parent/.style={draw,-latex},
  level 1/.style={sibling distance=8cm},
  level 2/.style={sibling distance=4cm},
  level 3/.style={sibling distance=2cm}]
\node {$\mathcal{O}_{0,1}$}
  child {node {$\mathcal{O}_{1,1}$}
    child {node {$\mathcal{O}_{2,1}$}
      child {node {$\sigma \circ \rho^{-2}$}}
      child {node {$\sigma \circ \rho^{-1}$}}}
    child {node {$\mathcal{O}_{2,2}$}
      child {node {$\sigma$}}
      child {node {$\sigma \circ \rho^{1}$}}}}
  child {node {$\mathcal{O}_{1,2}$}
    child {node {$\mathcal{O}_{2,3}$}
      child {node {$\sigma \circ \rho^{2}$}}
      child {node {$\sigma \circ \rho^{3}$}}}
    child {node {$\mathcal{O}_{2,4}$}
      child {node {$\sigma \circ \rho^{4}$}}
      child {node {$\sigma \circ \rho^{5}$}}}};
\end{tikzpicture}
\end{center}
\end{example}

In short, the \textsc{SubStop} condition holds if the no–underrun condition holds for at least one dyadic sub-orbit obtained by iteratively halving the extension.

The \textsc{Stop} and \textsc{SubStop} rules together ensure that the final orbit $\calO$ is symmetric with respect to its initial point; in particular, the probability of reconstructing the same orbit from any other point in $\calO$ is identical. This symmetry is used later to establish reversibility.

\paragraph{State Selection.}

Once the final orbit $\mathcal{O}$ is constructed, the next state is selected by sampling from a categorical distribution over $\mathcal{O}$ with weights $\{ w(\tau) \}_{\tau \in \mathcal{O}}$.   More generally, given any countable set $S$ and a non-negative function $f: S \to \mathbb{R}_{\ge 0}$ satisfying $\sum_{x \in S} f(x) < \infty$, the \emph{categorical distribution} on $S$ weighted by $f$ is the probability distribution defined by:
\begin{equation}\label{eq:cat}
    \mathbb{P}(X = x) = \frac{f(x)}{\sum_{y \in S} f(y)} \quad \text{for all } x \in S.
\end{equation}
We write this as $ X \sim \cat(S, f) $.

\paragraph{Implementation.}
This part summarizes the full algorithmic procedure for NURS. 
The sampler constructs an orbit by performing a randomized sequence of forward and backward doublings along a randomly chosen direction, and then selects the next state through categorical resampling (Algorithm~\ref{alg:nurs-transition}). 
Orbit expansion is managed by the subroutine \textsc{ExtendOrbit} (Algorithm~\ref{alg:extend-orbit}), while termination is governed by two complementary stopping rules. 
The \textsc{Stop} rule (Algorithm~\ref{alg:stop}) evaluates whether the current orbit satisfies the no–underrun condition, indicating that further expansion would add negligible total weight. 
The \textsc{SubStop} rule (Algorithm~\ref{alg:sub-stop}) checks whether any dyadic sub-orbit of the proposed extension already meets the same condition; in that case, the extension is rejected.  
Together, these conditions ensure that every admissible orbit could have been generated from any of its constituent states with equal probability.

\begin{algorithm}[ht]
\caption{\textsc{NURS} Transition Kernel: $\textsc{NURS}\!\left(\sigma,\varepsilon,M,w,q\right)$}
\label{alg:nurs-transition}
\textbf{Inputs:} current state $\sigma \in S_n$; maximum doublings $M \in \mathbb{N}$; threshold $\varepsilon \in (0,1)$; unnormalized weight function $w: S_n \to \mathbb{R}_{\ge 0}$; direction law $q(\cdot \mid \sigma)$ on $S_n$.

\textbf{Output:} next state $\sigma' \in S_n$.

\hrule
\begin{algorithmic}[1]
\State sample $\rho \sim q(\cdot \mid \sigma)$
\State sample $B \sim \Unif(\{0,1\}^M)$
\State initialize $a \gets 0$, $b \gets 0$, and $\mathcal{O} \gets (\sigma)$
\For{$j=1$ to $M$}
  \State $n_j \gets 2^{j-1}$
  \If{$B_j = 1$} \Comment{forward doubling}
    \State $\mathcal{O}^{\mathrm{ext}} \gets \textsc{ExtendOrbit}(\sigma \circ \rho^{\,b}, \rho, \, n_j)$
  \Else \Comment{backward doubling}
    \State $\mathcal{O}^{\mathrm{ext}} \gets \textsc{ExtendOrbit}(\sigma \circ \rho^{\,a}, \rho, \,-n_j)$
  \EndIf
  \If{\textsc{SubStop}$(\mathcal{O}^{\mathrm{ext}},\varepsilon,w)$}
    \State \textbf{break} \Comment{reject extension; terminate with current $\mathcal{O}$}
  \EndIf
  \If{$B_j = 1$}
    \State $\mathcal{O} \gets \mathcal{O} \odot \mathcal{O}^{\mathrm{ext}}$, \quad $b \gets b + n_j$
  \Else
    \State $\mathcal{O} \gets \mathcal{O}^{\mathrm{ext}} \odot \mathcal{O}$, \quad $a \gets a - n_j$
  \EndIf
  \If{\textsc{Stop}$(\mathcal{O},\varepsilon,w)$}
    \State \textbf{break} \Comment{accept extended orbit; terminate}
  \EndIf
\EndFor
\State sample $\sigma' \sim \cat(\mathcal{O},w)$
\State \Return $\sigma'$
\end{algorithmic}
\end{algorithm}

\begin{algorithm}[ht]
\caption{\textsc{ExtendOrbit}$(\sigma,\rho,n)$}
\label{alg:extend-orbit}
\textbf{Inputs:} starting state $\sigma \in S_n$; direction $\rho \in S_n$; steps $n \in \mathbb{Z}$.

\textbf{Output:} ordered list of $|n|$ consecutive elements along the $\rho$-orbit from $\sigma$.

\hrule
\begin{algorithmic}[1]
\If{$n>0$} \State \Return $(\sigma\circ\rho^{\,1},\ \sigma\circ\rho^{\,2},\ \dots,\ \sigma\circ\rho^{\,n})$
\Else \State \Return $(\sigma\circ\rho^{\,n},\ \dots,\ \sigma\circ\rho^{-2},\ \sigma\circ\rho^{-1})$
\EndIf
\end{algorithmic}
\end{algorithm}

\begin{algorithm}[ht]
\caption{\textsc{Stop}$(\mathcal{O},\varepsilon,w)$}
\label{alg:stop}
\textbf{Inputs:} orbit $\mathcal{O}\subseteq S_n$; threshold $\varepsilon\in(0,1)$; weight function $w:S_n\to\mathbb{R}_{\ge0}$.

\textbf{Output:} \texttt{True} iff $\mathcal{O}$ satisfies the no–underrun condition.

\hrule
\begin{algorithmic}[1]
\State let $\sigma^{\mathrm{L}}$, $\sigma^{\mathrm{R}}$ be the leftmost and rightmost elements of $\mathcal{O}$
\State \Return $\max\!\bigl\{w(\sigma^{\mathrm{L}}),\,w(\sigma^{\mathrm{R}})\bigr\} \le \varepsilon \sum_{\breve{\sigma}\in\mathcal{O}} w(\breve{\sigma})$
\end{algorithmic}
\end{algorithm}

\begin{algorithm}[ht]
\caption{\textsc{SubStop}$(\mathcal{O},\varepsilon,w)$}
\label{alg:sub-stop}
\textbf{Inputs:} orbit $\mathcal{O}\subseteq S_n$ with $|\mathcal{O}|$ a power of two; threshold $\varepsilon\in(0,1)$; weight function $w:S_n\to\mathbb{R}_{\ge0}$.

\textbf{Output:} \texttt{True} iff $\mathcal{O}$ or one of its dyadic sub-orbits satisfies the no–underrun condition.

\hrule
\begin{algorithmic}[1]
\If{$|\mathcal{O}|<2$} \State \Return \texttt{False} \EndIf
\State write $\mathcal{O}=\mathcal{O}^{\mathrm{L}}\odot \mathcal{O}^{\mathrm{R}}$ with $|\mathcal{O}^{\mathrm{L}}|=|\mathcal{O}^{\mathrm{R}}|=|\mathcal{O}|/2$
\State \Return \textsc{Stop}$(\mathcal{O},\varepsilon,w)$ \textbf{ or } \textsc{SubStop}$(\mathcal{O}^{\mathrm{L}},\varepsilon,w)$ \textbf{ or } \textsc{SubStop}$(\mathcal{O}^{\mathrm{R}},\varepsilon,w)$
\end{algorithmic}
\end{algorithm}

\clearpage

\section{Reversibility of NURS}

\label{sec:nurs-reversibility}

Reversibility is a fundamental property of many MCMC methods, ensuring that the target distribution remains invariant under the transition kernel of the Markov chain. In this section, we review a general auxiliary-variable framework for constructing reversible Markov chains using measure-preserving involutions, known as Gibbs Self-Tuning (GIST) \cite{BouRabeeCarpenterMarsden2024}. We then show the NURS transition is a rejection-free instance of this framework.  As a consequence, NURS inherits detailed balance with respect to its target distribution without requiring an explicit Metropolis–Hastings correction.

The GIST framework builds on the classical auxiliary-variable formulation of Andersen and Diaconis \cite[Section 4.1]{Diaconis2007HitandRun}, where auxiliary variables are introduced, updated, and then discarded at the end of each transition step. By contrast, the second generalization in \cite[Section~4.2]{Diaconis2007HitandRun}, which we do not pursue here, covers a broader class of ``lifted'' samplers in which auxiliary variables are retained across transitions. Additionally, we do not treat data augmentation approaches for latent or missing variables and on mixture representations (e.g., \cite{AlbertChib1993,TannerWong1987,VanDykMeng2001,PolsonScottWindle2013}), and instead focus on the involution-based subclass most directly connected to NURS.

\subsection{Gibbs Self-Tuning on Discrete State Spaces}
\label{sec:gist-discrete}

Let $\mathbb{S}$ be a countable state space with target distribution $\mu$ and let $\mathbb{V}$ denote a (possibly multi-dimensional) auxiliary variable space.  
The idea of GIST is to construct a reversible Markov kernel on $\mathbb{S}$ by introducing an auxiliary variable $v\in\mathbb{V}$, resampling $v$ conditionally on the current state at each iteration, and applying an involution on the augmented space $\mathbb{A} = \mathbb{S} \times \mathbb{V}$. In the discrete setting, the reference measure on $\mathbb{A}$ is the product of counting measures on $\mathbb{S}$ and $\mathbb{V}$.  
Any bijection on $\mathbb{A}$ automatically preserves this measure, so an involution $\Psi:\mathbb{A}\to\mathbb{A}$ is measure-preserving by definition.

Formally, let $p_a(v \mid \theta)$ be a conditional probability mass function on~$\mathbb{V}$ given $\theta\in\mathbb{S}$.  
The extended target distribution on $\mathbb{A}$ is
\[
\hat{\mu}(\theta,v) = \mu(\theta)\, p_a(v \mid \theta),
\]
interpreted as a probability mass function on $\mathbb{A}$ with respect to counting measure.  
Because both spaces are discrete, all transitions are defined in terms of probability masses rather than densities.

Below, we index Markov chain iterations by $k\in\mathbb N$ (to avoid conflict with $n$ in $S_n$).  Given a current state $\theta_k\in\mathbb{S}$, the conditional law $p_a(\cdot\mid \theta_k)$, and a measurable involution $\Psi:\mathbb{A}\to\mathbb{A}$, one GIST transition is defined as follows:

\begin{enumerate}
\item[\textbf{(1)}] \textbf{Auxiliary-variable update.}  
Draw $v_k \sim p_a(\cdot\mid\theta_k)$.

\item[\textbf{(2)}] \textbf{Involution-based proposal.}  
Compute $(\tilde{\theta}_{k+1},\tilde{v}_{k+1}) = \Psi(\theta_k,v_k)$.

\item[\textbf{(3)}] \textbf{Metropolis correction.}  
Accept the proposal with probability
\[
\alpha(\theta_k,v_k)
= \min\!\left(1,\,
\frac{\hat{\mu}(\tilde{\theta}_{k+1},\tilde{v}_{k+1})}
{\hat{\mu}(\theta_k,v_k)} \right),
\]
and set
\[
\theta_{k+1} =
\begin{cases}
\tilde{\theta}_{k+1}, & \text{with probability } \alpha(\theta_k,v_k),\\[4pt]
\theta_k, & \text{otherwise.}
\end{cases}
\]
\end{enumerate}

Since the auxiliary variable is resampled and then discarded at each step, the resulting chain on~$\mathbb{S}$ is marginally Markovian.  
When $\Psi$ is an involution, the induced transition kernel satisfies detailed balance with respect to~$\mu$.
\medskip

\begin{theorem}[Reversibility of GIST transitions \cite{BouRabeeCarpenterMarsden2024}]
\label{thm:AVM_reversibility_discrete}
Let $\mathbb{A} = \mathbb{S} \times \mathbb{V}$ be a discrete augmented space, and let $\Psi:\mathbb{A}\to\mathbb{A}$ be an involution, i.e., $\Psi^2=\mathrm{id}$.  
Then the Markov chain on~$\mathbb{S}$ defined by the GIST transition above is reversible with respect to~$\mu$.
\end{theorem}

\begin{proof}[Sketch of proof]
Let $\mathbb{A} = \mathbb{S} \times \mathbb{V}$ with counting measure as reference.  
The proposal on $\mathbb{A}$ is the deterministic pushforward
\[
r\big((\theta,v),\cdot\big) = \delta_{\Psi(\theta,v)}(\cdot)
\]
after refreshing $v \sim p_a(\cdot \mid \theta)$.  
Since $\Psi$ is an involution, it is a bijection that preserves the reference measure on $\mathbb{A}$ (Jacobian equal to~1 in the discrete sense).  
In particular,
\[
r\big((\theta,v),(\theta',v')\big)
= \mathbf{1}\!\left\{(\theta',v')=\Psi(\theta,v)\right\}
= \mathbf{1}\!\left\{(\theta,v)=\Psi(\theta',v')\right\}
= r\big((\theta',v'),(\theta,v)\big),
\]
so the proposal is \emph{symmetric} on $\mathbb{A}$.  
Therefore, the standard Metropolis--Hastings acceptance rule on $\mathbb{A}$ with target 
$\hat\mu(\theta,v) = \mu(\theta)\,p_a(v\mid\theta)$ satisfies detailed balance:
\[
\hat\mu(\theta,v)\, r\big((\theta,v),(\theta',v')\big)\, 
\alpha\big((\theta,v),(\theta',v')\big)
=
\hat\mu(\theta',v')\, r\big((\theta',v'),(\theta,v)\big)\,
\alpha\big((\theta',v'),(\theta,v)\big).
\]
Hence the Markov kernel on $\mathbb{A}$ is reversible with respect to~$\hat\mu$.  
Marginalizing over $v$ yields reversibility of the induced kernel on $\mathbb{S}$ with respect to~$\mu$. A full proof appears in~\cite{BouRabeeCarpenterMarsden2024}.
\end{proof}

The GIST framework offers a systematic way to construct reversible Markov kernels using measure-preserving involutions on augmented state spaces.   In what follows, we show that the discrete No-Underrun Sampler (NURS) fits exactly within this framework: it admits a natural involutive map~$\Psi$ on the extended space that preserves the joint law.  
Consequently, the Metropolis correction is trivial, and the resulting transition kernel is both rejection-free and reversible with respect to the Mallows distribution.

\subsection{NURS as a Rejection-Free GIST Sampler}

With Theorem~\ref{thm:AVM_reversibility_discrete} in hand, we now verify that the NURS transition kernel is reversible by expressing it as a special case of the GIST framework.

\medskip

\begin{theorem}
\label{thm:main-reversibility}
The transition kernel of  NURS  is reversible with respect to the Mallows distribution \eqref{eq:mallows}.
\end{theorem}

\begin{proof}
We apply Theorem~\ref{thm:AVM_reversibility_discrete}, which guarantees reversibility for any GIST sampler.
To do so, we embed the NURS transition within the GIST formalism.

\smallskip
\textbf{GIST embedding.}
The augmented state consists of the current permutation
$\sigma\in S_n$ together with auxiliary variables
\[
(\rho,m,b,i),
\]
where:
\begin{itemize}
  \item $\rho\in S_n$ is the sampled direction;
  \item $m\in\mathbb N$ is the number of doublings;
  \item $b\in\{0,\dots,2^m-1\}$ is the rightmost index of the final orbit;
  \item $i\in\{a,\dots,b\}$ is the selected index, with $a=b-2^m+1$.
\end{itemize}
The orbit itself is not an independent coordinate: it is deterministically
given by
\[
\mathcal O = (\sigma\circ\rho^k)_{k=a}^b .
\]
Thus the augmented space is
\[
\mathbb A := S_n \times S_n \times \mathbb N \times \mathbb N \times \mathbb Z ,
\]
with the understanding that all probability kernels below are implicitly
restricted to admissible tuples produced by the NURS construction (i.e.,
those consistent with the randomized doubling and stopping rules).

Each NURS transition from a state $\sigma \in S_n$ proceeds in two steps:
\begin{enumerate}
  \item \textbf{Auxiliary variable refresh.} Given the current state $\sigma \in S_n$, we:
  \begin{itemize}
    \item draw a random direction $\rho \sim q( \cdot \mid \sigma)$ according to an \emph{orbit-equivariant direction law} (see Definition~\ref{def:direction-law}),
    \item construct an orbit $\mathcal{O} = (\sigma \circ \rho^k)_{k=a}^b$ by randomized doubling,
    \item select an index $i \in \{a, \dots, b\}$ from a categorical distribution with weights equal to $w(\sigma \circ \rho^i)$ where $w$ is the weight function defined in \eqref{eq:weights}.
  \end{itemize}
  
  \item \textbf{Involution.} Apply the mapping $\Psi: \mathbb{A} \to \mathbb{A}$ defined by  \begin{equation}
  \Psi(\sigma, \rho, m, b, i) = (\sigma \circ \rho^i, \rho, m, b - i, -i).
  \end{equation}
The next state is set to $\sigma' = \sigma \circ \rho^i$, and the auxiliary variables are discarded.
\end{enumerate}
No Metropolis correction is required, since as shown below, $\Psi$ preserves the joint distribution on~$\mathbb{A}$.

The map $\Psi$ re-centers the orbit so that the selected index $i$ becomes the new origin. Lemma~\ref{lem:involution} confirms that $\Psi$ is an involution on $\mathbb{A}$.   In particular, $\Psi$ maps admissible augmented states to admissible augmented
states, since the NURS stopping and sub-stopping rules are invariant under
re-centering of the orbit.

Lemmas~\ref{lem:orbit_kernel} and \ref{lem:index_kernel} characterize the conditional distributions used in the auxiliary variable refresh step: the orbit selection kernel $p_{\orbit}$ gives the conditional distribution of the orbit parameters $(m, b)$ given $(\sigma, \rho)$, and the index selection kernel $p_{\ind}$ gives the conditional distribution of index $i$ given $(\sigma, \rho, m, b)$. Together with the Mallows distribution and the direction distribution on $S_n$, these define the extended target on $\mathbb{A}$:
\begin{equation} \label{eq:joint}
    p_{\text{joint}}(z) = P_{\beta}(\sigma)
\times
q(\rho \mid \sigma)
\times
p_{\orbit}\bigl(m, b \mid \sigma,\rho\bigr)
\times
p_{\ind}(i \mid \sigma, \rho, m, b).
\end{equation}

The key remaining step is to verify that $\Psi$ preserves this joint distribution.  
Lemma~\ref{lem:invariance} establishes precisely this invariance:
\[
p_{\text{joint}} \circ \Psi = p_{\text{joint}}.
\] 
As a result, NURS is rejection-free: the proposal $\sigma \circ \rho^i$ is always accepted and becomes the next state of the chain. With this invariance in place, all conditions of Theorem~\ref{thm:AVM_reversibility_discrete} are satisfied, and we conclude that the NURS transition kernel is reversible with respect to the Mallows distribution. \end{proof}

\begin{lemma} \label{lem:involution}
The mapping $\Psi$ is an involution on $\mathbb{A}$.
\end{lemma}

\begin{proof}
Let $z = (\sigma, \rho, m, b, i) \in \mathbb{A}$. By definition of $\Psi$, we have
\[
\Psi(z) = (\sigma \circ \rho^i, \rho, m, b - i, -i).
\]
Applying $\Psi$ again gives
\[
\Psi(\Psi(z)) = (\sigma \circ \rho^i \circ \rho^{-i}, \rho, m, b - i - (-i), -(-i)) = (\sigma, \rho, m, b, i) = z.
\]
Hence, $\Psi \circ \Psi = \mathrm{id}$ on $\mathbb{A}$, and $\Psi$ is an involution.
\end{proof}

Below, $M \in \mathbb{N}$ denotes the maximum number of doublings, $\varepsilon > 0$ denotes the density threshold used in the no-underrun condition from Definition~\ref{defn:no-underrun}, $w$ is the unnormalized Mallows density defined in~\eqref{eq:weights}, and $\textsc{Stop}$ and $\textsc{SubStop}$ refer to the procedures described in Algorithms~\ref{alg:stop} and~\ref{alg:sub-stop}, respectively.

\medskip

\begin{lemma} \label{lem:orbit_kernel}
Let $(m, b)$ specify a NURS orbit
\[
\mathcal{O} = (\sigma \circ \rho^i)_{i = a}^b,
\qquad a = b - 2^m + 1.
\]
Then the orbit selection kernel of NURS satisfies
\[
p_{\orbit}(m, b \mid \sigma, \rho)
=
\frac{1}{2^m}
\cdot
\mathbf{1}\Biggl(
\Phi(\mathcal{O})
\;\wedge\;
\bigwedge_{s = 1}^{m} \; \bigwedge_{t = 1}^{2^s}
\neg \textsc{Stop}\!\left(\mathcal{O}_{s,t}, \varepsilon, w\right)
\Biggr),
\]
where $\{\mathcal{O}_{s,t}\}$ are the dyadic sub-orbits of $\mathcal{O}$
(Definition~\ref{def:dyadic-suborbit}), and where $\Phi(\mathcal{O})$ holds if and only if
\[
\textsc{Stop}(\mathcal{O}, \varepsilon, w)
\;\vee\;
\textsc{SubStop}(\mathcal{O}_+^{\ext}, \varepsilon, w)
\;\vee\;
\textsc{SubStop}(\mathcal{O}_-^{\ext}, \varepsilon, w)
\;\vee\;
|\mathcal{O}| = 2^M.
\]
The forward and backward extensions are defined by
\[
\mathcal{O}_+^{\ext} = (\sigma \circ \rho^i)_{i = b+1}^{b+2^m},
\qquad
\mathcal{O}_-^{\ext} = (\sigma \circ \rho^i)_{i = a - 2^m}^{a - 1}.
\]
\end{lemma}

This result follows directly from the randomized doubling scheme used in NURS. At each doubling, the algorithm chooses to double forward or backward with equal probability and hence a factor of $1/2^m$ for each orbit of length $2^m$. The stopping criterion for the final orbit is described by the logical condition above. 

\begin{proof}
Fix $\sigma\in S_n$ and a direction $\rho$.
During a NURS transition, the orbit is constructed by a randomized doubling
procedure. At each doubling level $j=1,\dots,m$, the algorithm chooses to extend
the current orbit forward or backward, each with equal probability $1/2$.

Let $\mathcal{O} = (\sigma \circ \rho^i)_{i=a}^b$ be a final orbit of length $2^m$,
with $a=b-2^m+1$. The rightmost index $b$ uniquely determines a binary sequence
$B=(B_1,\dots,B_m)\in\{0,1\}^m$ via
\[
b = \sum_{j=1}^m B_j\,2^{j-1},
\]
where $B_j=1$ (resp.\ $0$) indicates that the $j$-th doubling was performed in
the forward (resp.\ backward) direction; see Figure~\ref{fig:binary}.
Since each doubling direction is chosen independently and uniformly, the
probability of generating this specific sequence $B$, and hence the orbit
$\mathcal{O}$, is exactly $2^{-m}$.

The orbit $\mathcal{O}$ is selected by the NURS procedure if and only if the
doubling construction halts at level $m$.
By Algorithm~\ref{alg:nurs-transition}, this occurs precisely when the condition
$\Phi(\mathcal{O})$ holds, i.e., when at least one of the following events occurs:
\begin{itemize}
\item $\mathcal{O}$ itself satisfies the \textsc{Stop} condition;
\item extending $\mathcal{O}$ forward or backward would trigger the
      \textsc{SubStop} condition;
\item the maximum number of doublings $M$ has been reached.
\end{itemize}

Moreover, since the orbit construction starts from a single state and proceeds
incrementally through the dyadic doubling tree, every dyadic sub-orbit arising
during the construction has already been tested and rejected by the stopping
rule. Consequently, no dyadic sub-orbit of $\mathcal{O}$ satisfies the
\textsc{Stop} condition.

Therefore, the joint probability of selecting $(m,b)$ is equal to $2^{-m}$ if
$\Phi(\mathcal{O})$ holds and no dyadic sub-orbit satisfies \textsc{Stop}, and is
zero otherwise. This yields the claimed expression for
$p_{\orbit}(m,b\mid\sigma,\rho)$.
\end{proof}

\begin{lemma} \label{lem:index_kernel}
Let $(\sigma,\rho,m,b)$ specify a NURS orbit
\[
\mathcal{O}=(\sigma\circ\rho^j)_{j=a}^b,
\qquad a=b-2^m+1.
\]
Conditional on $(\sigma,\rho,m,b)$, the index selected by NURS satisfies
\[
p_{\ind}(i \mid \sigma, \rho, m, b)
=
\frac{P_{\beta}(\sigma \circ \rho^i)}
{\sum_{j=a}^{b} P_{\beta}(\sigma \circ \rho^j)}
=
\frac{w(\sigma \circ \rho^i)}
{\sum_{j=a}^{b} w(\sigma \circ \rho^j)},
\qquad i\in\{a,\dots,b\},
\]
and $p_{\ind}(i \mid \sigma, \rho, m, b)=0$ for $i\notin\{a,\dots,b\}$.
\end{lemma}

\begin{proof}
    By definition of the NURS algorithm, the index $i$ is drawn from the categorical distribution over $a:b$ with unnormalized weights $(P_\beta(\sigma \circ \rho^j))_{j=a}^b$. Therefore, the selection probability is
\[
p_{\ind}(i \mid \sigma, \rho, m, b) 
= \frac{P_{\beta}(\sigma \circ \rho^i)}{\sum_{j = a}^b P_{\beta}(\sigma \circ \rho^j)} = \frac{w(\sigma \circ \rho^i)}{\sum_{j = a}^b w(\sigma \circ \rho^j)} , 
\]
as claimed.
\end{proof}

\begin{lemma} \label{lem:invariance}
For every $z\in\mathbb A$ with $p_{\joint}(z)>0$, it holds: $p_{\joint} (\Psi(z)) = p_{\joint}(z)$.
\end{lemma}

\begin{proof}
Fix $z=(\sigma,\rho,m,b,i)\in\mathbb{A}$ such that
\[
p_{\joint}(z)>0.
\]
Equivalently, the orbit
\[
\mathcal{O}=(\sigma\circ\rho^j)_{j=a}^b,
\qquad a=b-2^m+1,
\]
is the terminal orbit produced by the NURS doubling procedure, meaning that
\[
\Phi(\mathcal{O})
:= \textsc{Stop}(\mathcal{O},\varepsilon,w)
\ \vee\ 
\textsc{SubStop}(\mathcal{O}_+^{\ext},\varepsilon,w)
\ \vee\ 
\textsc{SubStop}(\mathcal{O}_-^{\ext},\varepsilon,w)
\ \vee\ 
|\mathcal{O}|=2^M
\]
holds, and no strict dyadic sub-orbit $\mathcal{O}_{s,t}\subsetneq\mathcal{O}$
satisfies the stopping rule.

By construction, the NURS stopping and sub-stopping rules depend only on the
collection of states visited along the orbit and their weights, and are
invariant under re-centering of the orbit index.
Consequently,
\[
p_{\orbit}(m,b-i \mid \sigma\circ\rho^i,\rho)
=
p_{\orbit}(m,b \mid \sigma,\rho).
\]

We now compute $p_{\joint}(\Psi(z))$. Using \eqref{eq:joint},
\begin{align*}
p_{\joint}(\Psi(z))
&\overset{\eqref{eq:joint}}{=}
P_{\beta}(\sigma\circ\rho^i)\ \cdot\ q(\rho\mid\sigma\circ\rho^i)\ \cdot\
p_{\orbit}(m,b-i\mid\sigma\circ\rho^i,\rho)\ \cdot\
p_{\ind}(-i\mid\sigma\circ\rho^i,\rho,m,b-i)\\
&\overset{\text{Lemma~\ref{lem:orbit_kernel}}}{=}
P_{\beta}(\sigma\circ\rho^i)\ \cdot\ q(\rho\mid\sigma\circ\rho^i)\ \cdot\
\frac{1}{2^m}\ \cdot\
p_{\ind}(-i\mid\sigma\circ\rho^i,\rho,m,b-i)\\
&\overset{\text{Lemma~\ref{lem:index_kernel}}}{=}
P_{\beta}(\sigma\circ\rho^i)\ \cdot\ q(\rho\mid\sigma\circ\rho^i)\ \cdot\
\frac{1}{2^m}\ \cdot\
\frac{P_{\beta}(\sigma)}{\sum_{j=a-i}^{\,b-i} P_{\beta}(\sigma\circ\rho^{\,i+j})}\,.
\end{align*}
Here the denominator
\[
\sum_{j=a-i}^{\,b-i} P_{\beta}(\sigma\circ\rho^{\,i+j})
\]
is the normalization constant of the index-selection kernel
$p_{\ind}(-i \mid \sigma\circ\rho^i,\rho,m,b-i)$ associated with the
re-centered orbit.

Next, invoke the \emph{orbit–equivariance} of the direction law (Definition~\ref{def:direction-law}):
\[
q(\,\cdot\mid\sigma\,)=q(\,\cdot\mid\sigma\circ\rho^k\,)\quad\text{for all }k,
\]
and in particular $q(\rho\mid\sigma\circ\rho^i)=q(\rho\mid\sigma)$.

Also, change variables in the denominator via $j' = i+j$, which maps
$j\in[a-i,b-i]$ to $j'\in[a,b]$. Hence
\[
\sum_{j=a-i}^{b-i} P_{\beta}(\sigma\circ\rho^{\,i+j})
=\sum_{j'=a}^{b} P_{\beta}(\sigma\circ\rho^{\,j'}).
\]
Substituting these two identities gives
\begin{align*}
p_{\joint}(\Psi(z))
&= P_{\beta}(\sigma\circ\rho^i)\ \cdot\ q(\rho\mid\sigma)\ \cdot\ \frac{1}{2^m}\ \cdot\
\frac{P_{\beta}(\sigma)}{\sum_{j=a}^{b} P_{\beta}(\sigma\circ\rho^{\,j})} \\
&= P_{\beta}(\sigma)\ \cdot\ q(\rho\mid\sigma)\ \cdot\ \frac{1}{2^m}\ \cdot\
\frac{P_{\beta}(\sigma\circ\rho^i)}{\sum_{j=a}^{b} P_{\beta}(\sigma\circ\rho^{\,j})}
\overset{\eqref{eq:joint}}{=} p_{\joint}(z).
\end{align*}
Therefore $p_{\joint}(\Psi(z)) = p_{\joint}(z)$ for every $z$ with $p_{\joint}(z)>0$, as claimed.
\end{proof}

\begin{figure}[t]
  \centering
\begin{tikzpicture}[
  grow = down,
  scale=0.8,
  edge from parent/.style = {draw},
  level 1/.style = {sibling distance=8cm, level distance=1.5cm},
  level 2/.style = {sibling distance=4cm, level distance=1.5cm},
  level 3/.style = {sibling distance=2cm, level distance=1.5cm},
  every node/.style = {align=center, font=\small}
]

\node {$\textcolor{red}{(\sigma\circ\rho^{-2}, \dotsc, \sigma\circ\rho^{5})}$}
  child {
    node {$\textcolor{red}{(\sigma\circ\rho^{-2},\sigma\circ\rho^{-1}, \sigma,\sigma\circ\rho^{1})}$}
    child {
      node {$(\sigma\circ\rho^{-2},\sigma\circ\rho^{-1})$}
      child {
        node {$\sigma\circ\rho^{-2}$}
        edge from parent node[left] {1}
      }
      child {
        node {$\sigma\circ\rho^{-1}$}
        edge from parent node[right] {0}
      }
      edge from parent node[left] {1}
    }
    child {
      node {$\textcolor{red}{(\sigma,\sigma\circ\rho^{1})}$}
      child {
        node {$\textcolor{red}{\sigma}$}
        edge from parent node[left,red] {\textbf{1}}
      }
      child {
        node {$\sigma\circ\rho^{1}$}
        edge from parent node[right] {0}
      }
      edge from parent node[right,red] {\textbf{0}}
    }
    edge from parent node[left,above,red] {\textbf{1}}
  }
  child {
    node {$(\sigma\circ\rho^{2},\sigma\circ\rho^{3}, \sigma\circ\rho^{4},\sigma\circ\rho^{5})$}
    child {
      node {$(\sigma\circ\rho^{2},\sigma\circ\rho^{3})$}
      child {
        node {$\sigma\circ\rho^{2}$}
        edge from parent node[left] {1}
      }
      child {
        node {$\sigma\circ\rho^{3}$}
        edge from parent node[right] {0}
      }
      edge from parent node[left] {1}
    }
    child {
      node {$(\sigma\circ\rho^{4},\sigma\circ\rho^{5})$}
      child {
        node {$\sigma\circ\rho^{4}$}
        edge from parent node[left] {1}
      }
      child {
        node {$\sigma\circ\rho^{5}$}
        edge from parent node[right] {0}
      }
      edge from parent node[right] {0}
    }
    edge from parent node[right,above] {0}
  }
;
\end{tikzpicture}

\caption{\textbf{Binary-tree representation of the dyadic orbit construction.}   Starting from $\sigma$, each bit $B_j\in\{0,1\}$ specifies a forward ($1$) or backward ($0$) doubling at level $j$. 
  The concatenation $B=(B_1,\dots,B_m)$ determines the rightmost index $b=\sum_{j=1}^m B_j\,2^{\,j-1}$ and hence the orbit interval $[a,b]$ with $a=b-2^m+1$. 
  The highlighted path illustrates the case $B=(1,0,1)$, yielding rightmost index $b=5$ and orbit interval $[a,b]=[-2,5]$. } \label{fig:binary}
\end{figure}
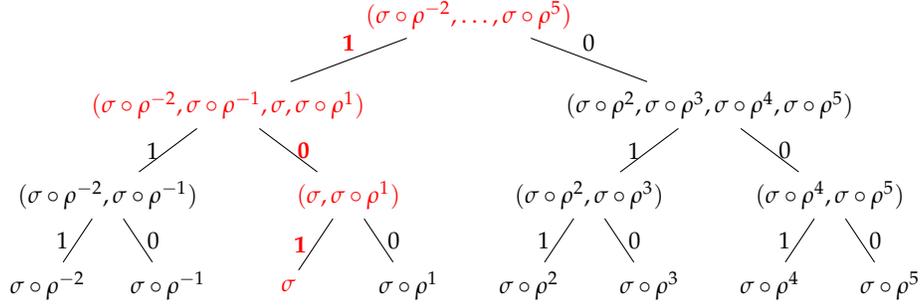

\section{Mixing Time Analysis}
\label{sec:coupling}

This section derives quantitative mixing-time bounds for the discrete
No-Underrun Sampler (NURS) targeting the Mallows$(d,\sigma_0)$ distribution,
using the path-coupling method. Our analysis is conducted on the transposition
Cayley graph of $S_n$, endowed with the Cayley distance
$d_{\mathrm{Cay}}$, so that neighboring states are pairs
$(\sigma,\sigma\circ\tau_{ij})$ differing by a single transposition. The central
task is to control the expected one-step evolution of this distance under a
suitably constructed coupling of the Markov chain.

Throughout this section, we fix a concrete \emph{shiftable} direction law and
analyze an idealized NURS transition kernel in which the orbit window is taken
to be the full order of the sampled direction $\eta$ and the
stopping and sub-stopping rules are suppressed. This idealized kernel coincides
with the NURS transition when the orbit window is taken to be of full length.

The argument proceeds as follows. First, the direction law is chosen so that,
whenever the sampled transposition index $(I,J)$ matches the edge $(i,j)$ under
consideration, the NURS orbits of the two coupled chains
coincide up to a deterministic index shift. On this \emph{aligned} event
$(I,J)=(i,j)$, we exploit the shift structure to construct a perfect one-step
coupling and quantify its contribution to contraction. On the complementary
\emph{mismatch} event $(I,J)\neq(i,j)$, we instead couple two orbit segments whose
corresponding points remain at Cayley distance one, using a maximal coupling of
the associated categorical distributions.

Combining these two cases yields an explicit edge-wise contraction bound for the
Cayley distance under any energy function satisfying a mild Lipschitz condition.
An application of the Path Coupling Theorem~\cite{BubleyDyer1997} then propagates
this local contraction to the entire state space $S_n$, leading to an
$O(n^2\log n)$ upper bound on the total-variation mixing time in a
high-temperature regime.

\subsection{Path Coupling Theorem}

We recall the path coupling theorem of Bubley and Dyer~\cite{BubleyDyer1997}.

\medskip

\begin{theorem}[Path coupling]\label{thm:path-coupling}
Let $G=(\mathcal{X},E)$ be a finite, connected, undirected graph whose edges
$\{u,v\}\in E$ are assigned lengths $l(u,v)\ge 1$, and let $d$ denote the
resulting shortest-path metric on $\mathcal{X}$. Let $K$ be a Markov kernel on
$\mathcal{X}$ with stationary distribution $\pi$. Suppose that for every edge
$\{u,v\}\in E$ there exists a coupling $(U,V)$ of $K(u,\cdot)$ and $K(v,\cdot)$
such that
\begin{equation}\label{eq:edge-contraction}
\EE_{u,v}\!\left[d(U,V)\right]
\;\le\;
d(u,v)\,e^{-\alpha},
\qquad \text{for some }\alpha>0.
\end{equation}
Then, for all $t\ge 0$,
\[
\max_{x\in\mathcal{X}}
\bigl\|K^t(x,\cdot)-\pi\bigr\|_{\TV}
\;\le\;
e^{-\alpha t}\,\diam(\mathcal{X},d),
\]
where $\diam(\mathcal{X},d):=\max_{x,y\in\mathcal{X}} d(x,y)$.
\end{theorem}

In our setting, the underlying graph $G=(\mathcal{X},E)$ is the transposition
Cayley graph of the symmetric group $S_n$. We take
\[
\mathcal{X} := S_n,
\qquad
\mathcal{T} := \{(i\,j) : 1 \le i < j \le n\},
\]
with edge set
\[
E
:= \bigl\{\,\{\sigma,\sigma\circ\tau\} : \sigma\in S_n,\ \tau\in\mathcal{T}\bigr\}.
\]
Each edge connects two permutations that differ by a single transposition and is
assigned unit length,
\[
l(\sigma,\sigma\circ\tau)=1.
\]
The induced path metric $\rho$ is therefore the Cayley distance
$d_{\mathrm{Cay}}$ with respect to the generating set $\mathcal{T}$. In
particular, the diameter of $(S_n,d_{\mathrm{Cay}})$ satisfies
\[
D_{\max}
:= \diam(S_n,d_{\mathrm{Cay}})
= \max_{\pi,\rho\in S_n} d_{\mathrm{Cay}}(\pi,\rho)
= n-1.
\]

\subsection{Shiftable Law and Kernel}

Throughout this section we work with \emph{shiftable} direction laws, which play a
central role in the mixing-time analysis.

\medskip

\begin{definition}[Shiftable directions]\label{def:Omegai}
For a transposition $\tau_{ij} = (i\,j)$ with $1 \le i < j \le n$, define the
associated direction set
\[
  \Omega_{ij}
  :=
  \Bigl\{
    \eta = \tau_{ij} h :
    h \in S_{[n]\setminus\{i,j\}},
    \text{ and every cycle of } h \text{ has odd length}
  \Bigr\},
\]
where
\[
  S_{[n]\setminus\{i,j\}}
  :=
  \bigl\{\, h \in S_n : h(i)=i,\ h(j)=j \,\bigr\}
  \cong S_{n-2}.
\]
A permutation $\eta \in S_n$ is called \emph{shiftable} if
$\eta \in \Omega_{ij}$ for some $1 \le i < j \le n$. A direction law $q$ on $S_n$
is called shiftable if its support is contained in
\[
  \Omega := \bigcup_{1 \le i < j \le n} \Omega_{ij}.
\]
\end{definition}

\medskip

For any $\eta \in \Omega_{ij}$ we may write $\eta = \tau_{ij} h$ with $h$ fixing
$i$ and $j$ and consisting only of odd-length cycles. Since $h$ commutes with
$\tau_{ij}$ and has odd order $m := \ord(h)$, it follows that
\[
  \ord(\eta) = 2m
  \qquad\text{and}\qquad
  \eta^{m} = \tau_{ij}.
\]
Accordingly, for any $\sigma \in S_n$ the full $\eta$-orbit has length $2m$ and
is given by
\[
  \bigl\{\, \sigma \eta^t : t = 0,1,\dots,2m-1 \,\bigr\}.
\]

\medskip

\begin{definition}[Shiftable kernel]\label{def:shiftable-kernel}
Let $K$ denote the NURS transition kernel induced by a shiftable direction law
$q$, and let $K^t$ denote its $t$-step iterate.
Given a current state $\sigma \in S_n$, one step of $K$ is constructed as
follows:
\begin{enumerate}
  \item Sample an index pair $(I,J)$ uniformly from
  $\{(i,j):1\le i<j\le n\}$, and then sample a direction
  $\eta \in \Omega_{IJ}$ according to a fixed, state-independent law.
  \item Let $2m := \ord(\eta)$ and set $W := \{0,1,\dots,2m-1\}$.
  Sample an index $r \in W$ with probability proportional to
  $w(\sigma \eta^{r})$, and update the state to $\sigma' = \sigma \eta^{r}$.
\end{enumerate}
For the coupling across an edge $(\sigma,\sigma\circ\tau_{ij})$ we use the same
random choices $((I,J),\eta)$ for both chains.
\end{definition}

\medskip

For a fixed edge $(\sigma,\sigma\circ\tau_{ij})$, it is convenient to write
\[
  X_t := \sigma \eta^{t},
  \qquad
  Y_t := \sigma\circ\tau_{ij}\,\eta^{t},
  \qquad t\in\mathbb{Z}.
\]
The relationship between the two orbits depends on whether the sampled index
pair $(I,J)$ coincides with the edge $(i,j)$ under consideration:
\begin{itemize}
  \item \emph{Aligned case} $(I,J)=(i,j)$.  
  In this case $\eta^m=\tau_{ij}$, and hence
  \[
    Y_t = X_{t+m},
  \]
  so the two orbits coincide up to a fixed index shift.
  \item \emph{Mismatch case} $(I,J)\neq(i,j)$.  
  In this case the two orbits need not align.  Nevertheless, for every $t\in\mathbb Z$,
\[
  X_t^{-1}Y_t
  =(\sigma\eta^{t})^{-1}\,(\sigma\circ\tau_{ij}\,\eta^{t})
  =\eta^{-t}\tau_{ij}\eta^{t},
\]
which is a transposition (a conjugate of $\tau_{ij}$). Consequently,
\[
  d_{\mathrm{Cay}}(X_t,Y_t)=1
  \qquad\text{for all }t,
\]
even though there is no fixed index shift relating the two orbits.
\end{itemize}

We analyze these two cases separately in
Sections~\ref{subsec:alignment} and~\ref{subsec:mismatch}.

\medskip

Finally, define the maximal \emph{cross-orbit diameter}
\[
D_{\mathrm{cross}}
:= \sup_{\sigma \in S_n}
   \max_{\eta \in \Omega}
   \max_{1 \le i < j \le n}
   \max_{s,t\in\{0,\dots,\ord(\eta)-1\}}
   d_{\mathrm{Cay}}\bigl(\sigma \eta^{s},\, \sigma \circ \tau_{ij} \eta^{t}\bigr).
\]
Clearly,
\[
D_{\mathrm{cross}} \;\le\; D_{\max} = \diam(S_n,d_{\mathrm{Cay}}) = n-1.
\]

\subsection{Mixing Result}

We now state the main quantitative result of the mixing-time analysis. The theorem
below establishes an explicit one-step contraction for the Cayley distance under
the shiftable NURS kernel, which in turn yields a total-variation mixing bound via
the path-coupling method. The result applies to an idealized version of the NURS
transition in which the orbit window is taken to be the full order of the sampled
direction and is stated under a mild Lipschitz regularity assumption on the
energy. All constants appearing in the contraction factor are made explicit.

\medskip

\begin{theorem}[Mixing time for the shiftable kernel]
\label{thm:main-mixing}
Let $K$ be the shiftable kernel from Definition~\ref{def:shiftable-kernel}, 
with stationary distribution $P_\beta$, and let $(I,J)$ be chosen uniformly from $\{(i,j):1\le i<j\le n\}$ at each step. Assume that the energy $E$ is Lipschitz with respect to the Cayley distance, i.e.
\begin{equation}
\label{eq:Lipschitz-again}
  |E(\pi)-E(\rho)|
  \;\le\;
  L_E\, d_{\mathrm{Cay}}(\pi,\rho)
  \qquad\text{for all }\pi,\rho\in S_n.
\end{equation}
In particular, whenever $\pi$ and $\rho$ differ by one transposition
($d_{\mathrm{Cay}}(\pi,\rho)=1$) we have $|E(\pi)-E(\rho)|\le L_E$.
Then for every adjacent pair $(\sigma,\sigma \circ \tau_{ij})$ there exists a one-step coupling $(U,V)$ of $K(\sigma,\cdot)$ and $K(\sigma \circ \tau_{ij},\cdot)$ such that
\[
  \mathbb{E}\bigl[d_{\mathrm{Cay}}(U,V)\bigr]
  \;\le\;
  \delta(\beta)
  \;:=\;
  \Bigl(1-\frac{2}{n(n-1)}\Bigr)
  \Bigl[\,1 + (D_{\mathrm{cross}}-1)\,\tanh(\beta L_E)\Bigr].
\]
When $\delta(\beta)<1$, the path–coupling condition \eqref{eq:edge-contraction} holds with $\rho=d_{\mathrm{Cay}}$ and $e^{-\alpha}=\delta(\beta)$. Consequently, for all $t\ge 0$,
\[
  \max_{\sigma\in S_n}
  \bigl\|K^t(\sigma,\cdot)-P_\beta\bigr\|_{\TV}
  \;\le\;
  D_{\max}\,\delta(\beta)^{\,t},
\]
and for every $\varepsilon\in(0,1)$ the $\varepsilon$–mixing time satisfies
\[
  t_{\mathrm{mix}}(\varepsilon)
  \;\le\;
  \frac{\log D_{\max} + \log(1/\varepsilon)}{-\log\delta(\beta)}.
\]
\end{theorem}

\begin{proof}
Fix an edge $(\sigma,\sigma \circ \tau_{ij})$. We construct a one-step coupling by sharing the random choices $((I,J),\eta)$ for both chains. There are two cases: the aligned case $(I,J)=(i,j)$, and the mismatch case $(I,J)\neq(i,j)$.  With probability $2/(n(n-1))$ we are in the aligned case $(I,J)=(i,j)$, and by Lemma~\ref{lem:aligned-perfect} we can couple the updates so that $d_{\mathrm{Cay}}(U,V)=0$. With the complementary probability $1-2/(n(n-1))$ we are in the mismatch case $(I,J)\neq(i,j)$, and Theorem~\ref{thm:mismatch-Ed} yields
\[
  \EE\bigl[d_{\mathrm{Cay}}(U,V)\,\big|\,(I,J)\neq(i,j)\bigr]
  \;\le\;
  1 + (D_{\mathrm{cross}}-1)\,\tanh(\beta L_E).
\]
Taking expectations over the choice of $(I,J)$ gives the bound
$\EE[d_{\mathrm{Cay}}(U,V)] \le \delta(\beta)$ as stated. Since
$d_{\mathrm{Cay}}(\sigma,\sigma \circ \tau_{ij})=1$, this is exactly the
edge-wise contraction condition in the path Theorem~\ref{thm:path-coupling} with factor $e^{-\alpha}=\delta(\beta)$. Applying Theorem~\ref{thm:path-coupling} with $\rho=d_{\mathrm{Cay}}$ and
$\diam(\mathcal{X})=D_{\max}$ yields the TV decay and mixing-time bound.
\end{proof}

\begin{remark}[Scale of contraction and temperature]
In the setting of Theorem~\ref{thm:main-mixing}, the one-step contraction factor is
\[
  \delta(\beta)
  :=
  \Bigl(1-\frac{2}{n(n-1)}\Bigr)
  \Bigl[\,1 + (D_{\mathrm{cross}}-1)\,\tanh(\beta L_E)\Bigr].
\]
Using $\tanh x \le x$ and $D_{\mathrm{cross}} \le n-1$, we obtain
\[
  1-\delta(\beta)
  \;\ge\;
  \frac{2}{n(n-1)} - (n-2)\,\beta L_E.
\]
Thus there exist constants $c_1,c_2>0$ such that whenever \(\beta \;\le\; \frac{c_1}{n^3\,L_E}\) (Table ~\ref{tab:beta-scales}), we have \( 1-\delta(\beta) \;\ge\; \frac{c_2}{n^2}.\) Plugging this into the mixing-time bound in Theorem~\ref{thm:main-mixing} yields in high-temperature regime $\beta = O\!\bigl(1/(n^3 L_E)\bigr)$.
\[
  t_{\mathrm{mix}}(\varepsilon)
  \;\le\;
  \frac{\log D_{\max} + \log(1/\varepsilon)}{1-\delta(\beta)}
  \;=\;
  O\!\bigl(n^2\bigl(\log n + \log(1/\varepsilon)\bigr)\bigr).
\]
\end{remark}

\begin{table}[H]
\centering
\begin{tabular}{@{}llll@{}}
\toprule
Energy $E(\pi)$ & $E_{\max}$ & Local jump $L_E$ (Cayley step) & $\beta$ scale \\
\midrule
Kendall--$\tau$
& $\displaystyle \binom{n}{2}$
& $\displaystyle 2(n-1)-1$
& $\Theta(n^{-4})$ \\[4pt]

Footrule ($L^1$)
& $\displaystyle \lfloor n^{2}/2\rfloor$
& $\displaystyle 2(n-1)$
& $\Theta(n^{-4})$ \\[4pt]

Spearman--$\rho$ ($L^2$--squared)
& $\displaystyle \frac{n^{3}-n}{3}$
& $\displaystyle 2(n-1)^{2}$
& $\Theta(n^{-5})$ \\[4pt]

Hamming ($\#\{i:\pi(i)\neq i\}$)
& $\displaystyle n$
& $\displaystyle 2$
& $\Theta(n^{-3})$ \\[4pt]

Ulam ($n-\mathrm{LIS}(\pi)$)
& $\displaystyle n-1$
& $\displaystyle 2$
& $\Theta(n^{-3})$ \\[4pt]

Cayley
& $\displaystyle n-1$
& $\displaystyle 1$
& $\Theta(n^{-3})$ \\
\bottomrule
\end{tabular}
\caption{$\beta$ scale for contraction}
\label{tab:beta-scales}
\end{table}

\begin{remark}[Implication for Barker transpositions]
The Barker chain with random transposition proposals is obtained as the degenerate case of the shiftable kernel in which each
$\Omega_{ij}$ is supported on the bare transposition $\tau_{ij}$ (so
$\ord(\tau_{ij})=2$ and the orbit has length $2$). Thus the path–coupling bound of Theorem~\ref{thm:main-mixing} applies directly to this case. We have
\[
  t_{\mathrm{mix}}^{\mathrm{Barker}}(\varepsilon)
  \;=\;
  \calO\!\bigl(n^2\bigl(\log n + \log(1/\varepsilon)\bigr)\bigr),
\]
so the Barker transposition dynamics mixes on the same $n^2\log n$ time scale as the shiftable direction NURS kernel. Moreover, in this specialization the quantity $D_{\mathrm{cross}}$ in Theorem~\ref{thm:mismatch-Ed} is at most $2$, so the contraction condition holds already for $\beta$ of order $1/(n^2 L_E)$ rather than $1/(n^3 L_E)$, i.e.\ on an $n$-times larger high-temperature window than in the general shiftable direction case.
\end{remark}

\subsection{Weight control along paired orbits}
\label{subsec:weight-control}

Let $E:S_n\to\mathbb{R}$ be an energy, and for fixed $\beta>0$ define the unnormalised
Mallows weights
\[
  w(\sigma) \;=\; \exp\bigl(-\beta\,E(\sigma)\bigr),\qquad \sigma\in S_n.
\]
We will control the total–variation distance between two categorical laws obtained by
normalising the weights along two ``paired'' orbit segments.

Let $\calI$ be a finite index set and consider two indexed families
\[
  \{\sigma_t\}_{t\in\calI},\quad \{\sigma'_t\}_{t\in\calI}\subset S_n.
\]
Write
\[
  w_t := w(\sigma_t),
  \qquad
  w'_t := w(\sigma'_t),
  \qquad
  Z := \sum_{s\in\calI} w_s,
  \qquad
  Z' := \sum_{s\in\calI} w'_s,
\]
and define the corresponding categorical distributions on $\calI$ by
\[
  p(t) := \frac{w_t}{Z},
  \qquad
  q(t) := \frac{w'_t}{Z'},
  \qquad t\in\calI.
\]

We assume a global Lipschitz condition for the energy in the Cayley metric:
\begin{equation}
  \label{eq:Lipschitz}
  |E(\pi)-E(\rho)|
  \;\le\;
  L_E\, d_{\mathrm{Cay}}(\pi,\rho)
  \quad\text{for all }\pi,\rho\in S_n,
\end{equation}
and a uniform proximity assumption for the paired orbit points:
\begin{equation}
  \label{eq:k-proximity}
  d_{\mathrm{Cay}}(\sigma_t,\sigma'_t) \;\le\; k
  \quad\text{for all } t\in\calI,
\end{equation}
for some fixed integer $k\ge1$.
In the coupling below we will apply this with $k=1$, but the statement is
no harder in the slightly more general form. We first record a small, self-contained lemma that will be used later.

\medskip

\begin{lemma}[TV bound under bounded log-likelihood ratio]
\label{lem:TV-llr}
Let $p$ and $q$ be probability mass functions on a finite set $\calI$ such that
\[
  e^{-\varepsilon}
  \;\le\;
  \frac{q(t)}{p(t)}
  \;\le\;
  e^{\varepsilon}
  \qquad
  \text{for all }t\in\calI
\]
for some $\varepsilon\ge0$.
Then
\[
  \TV(p,q) \;\le\; \tanh\!\bigl(\varepsilon/2\bigr).
\]
\end{lemma}

\begin{proof}
Fix $t\in\calI$.
If $q(t)\le p(t)$, then $q(t)\ge e^{-\varepsilon}p(t)$ and hence
\[
  p(t)+q(t)
  \;\le\;
  p(t) + e^{\varepsilon}q(t)
  \;=\;
  (1+e^{\varepsilon})\,q(t),
\]
so
\(
  \min\{p(t),q(t)\}=q(t)\ge (p(t)+q(t))/(1+e^{\varepsilon}).
\)
If instead $p(t)\le q(t)$ we obtain the same bound by symmetry.
Thus, for all $t$,
\[
  \min\{p(t),q(t)\}
  \;\ge\;
  \frac{p(t)+q(t)}{1+e^{\varepsilon}}.
\]
Summing over $t$ and using $\sum_t p(t)=\sum_t q(t)=1$ gives
\[
  \sum_{t}\min\{p(t),q(t)\}
  \;\ge\;
  \frac{2}{1+e^{\varepsilon}}.
\]
Since
\(
  \TV(p,q)
  = 1-\sum_t\min\{p(t),q(t)\},
\)
we conclude that
\[
  \TV(p,q)
  \;\le\;
  1-\frac{2}{1+e^{\varepsilon}}
  \;=\;
  \frac{e^{\varepsilon}-1}{e^{\varepsilon}+1}
  \;=\;
  \tanh\!\bigl(\varepsilon/2\bigr).
\]
\end{proof}

We now specialise this to the Mallows weights along two paired orbit segments.

\medskip

\begin{theorem}[Orbit TV bound for Mallows weights]
\label{thm:mallows-TV}
Assume \eqref{eq:Lipschitz} and \eqref{eq:k-proximity}.
Then the categorical laws $p$ and $q$ defined above satisfy
\[
  \TV(p,q)
  \;\le\;
  \tanh\!\bigl(\beta L_E k\bigr).
\]
In particular, when $k=1$ we have $\TV(p,q)\le\tanh(\beta L_E)$.
\end{theorem}

\begin{proof}
From \eqref{eq:Lipschitz} and \eqref{eq:k-proximity},
\[
  |E(\sigma'_t)-E(\sigma_t)|
  \;\le\;
  L_E\,d_{\mathrm{Cay}}(\sigma_t,\sigma'_t)
  \;\le\;
  L_E k,
\]
so
\[
  \bigl|\log w'_t - \log w_t\bigr|
  = \beta\,\bigl|E(\sigma_t)-E(\sigma'_t)\bigr|
  \;\le\;
  \beta L_E k
  \qquad\text{for all }t\in\calI.
\]
Summing over $t$ yields
\[
  e^{-\beta L_E k} Z
  \;\le\;
  Z'
  \;\le\;
  e^{\beta L_E k} Z,
  \qquad\text{so}\qquad
  \bigl|\log Z' - \log Z\bigr|
  \;\le\;
  \beta L_E k.
\]
For each $t$ we therefore have
\[
  \frac{q(t)}{p(t)}
  = \frac{w'_t/Z'}{w_t/Z}
  = \frac{w'_t}{w_t}\,\frac{Z}{Z'},
\]
and the two bounds above imply
\[
  \bigl|\log(q(t)/p(t))\bigr|
  \;\le\;
  \bigl|\log w'_t - \log w_t\bigr|
  + \bigl|\log Z - \log Z'\bigr|
  \;\le\;
  2\beta L_E k.
\]
Thus $e^{-\varepsilon}\le q(t)/p(t)\le e^{\varepsilon}$ for all $t$, with
$\varepsilon := 2\beta L_E k$.
Applying Lemma~\ref{lem:TV-llr} gives
\[
  \TV(p,q)
  \;\le\;
  \tanh\!\bigl(\varepsilon/2\bigr)
  \;=\;
  \tanh\!\bigl(\beta L_E k\bigr).
\]
\end{proof}

\subsection{Alignment Case}
\label{subsec:alignment}
We first consider the \emph{aligned case}, in which the sampled index pair
$(I,J)$ coincides with the edge $(i,j)$ under consideration. In this setting,
the two NURS orbits coincide up to a deterministic index shift, allowing for a
perfect one-step coupling.

\medskip

\begin{lemma}[Perfect coupling in the aligned case]\label{lem:aligned-perfect}
Fix $1\le i<j\le n$ and let $\eta\in\Omega_{ij}$ with $\ord(\eta)=2m$, so that
$\eta^{m}=\tau_{ij}$. For $\sigma\in S_n$ set $\sigma' := \sigma \circ \tau_{ij}$ and
recall the orbit notation
\[
  X_t := \sigma\eta^{t},\qquad
  Y_t := \sigma'\eta^{t},\qquad t\in\mathbb{Z}.
\]
Then for all $t\in\mathbb{Z}$,
\[
  Y_t = X_{t+m},
\]
so the two length-$2m$ orbits $\{X_t\}_{t=0}^{2m-1}$ and
$\{Y_t\}_{t=0}^{2m-1}$ coincide up to the fixed index shift $m$. In
particular, for any window $W\subset\{0,\dots,2m-1\}$ the categorical laws
of $X_T$ and $Y_{T'}$ with weights proportional to $w(X_t)$ and $w(Y_t)$
can be coupled with total variation distance~$0$ by taking $T$ and
$T' = T+m \pmod{2m}$.
\end{lemma}

\begin{proof}
Using $\eta^{m}=\tau_{ij}$ and the definition of $\sigma'$,
\[
  Y_t = \sigma'\eta^{t}
      = \sigma \circ \tau_{ij}\eta^{t}
      = \sigma\eta^{m}\eta^{t}
      = \sigma\eta^{t+m}
      = X_{t+m},
\]
for all $t\in\mathbb{Z}$. Thus $\{Y_t\}$ is just the cyclic shift of
$\{X_t\}$ by $m$. Given any index window $W\subset\{0,\dots,2m-1\}$, set
$T$ with law proportional to $w(X_t)$ on $W$, and define $T' = T+m \pmod{2m}$.
Then $X_T$ and $Y_{T'}$ have the same weighted law, so the corresponding categorical distributions are coupled perfectly, with total variation distance~$0$.
\end{proof}

\begin{figure}[H]
\centering
\begin{tikzpicture}[xscale=0.85, node distance=1.1cm]
  \tikzset{
    dot/.style={circle, fill=black, inner sep=2.5pt},
    dotgray/.style={circle, fill=gray!60, inner sep=2.5pt}
  }

  \node[dot,label=above:{$\scriptstyle\sigma$}] (sigmaTop) {};
  \node[dot,label=above:{$\scriptstyle\sigma \circ \eta^{1}$}] (t1) [right=of sigmaTop] {};
  \node[dot,label=above:{$\scriptstyle\sigma \circ \eta^{2}$}] (t2) [right=of t1] {};
  \node[dot,label=above:{$\scriptstyle\sigma \circ \eta^{3}$}] (t3) [right=of t2] {};
  \node[dot,label=above:{$\scriptstyle\sigma \circ \eta^{4}$}] (t4) [right=of t3] {};
  \node[dot,label=above:{$\scriptstyle\sigma \circ \eta^{5}$}] (t5) [right=of t4] {};

  \node[dot,label=above:{$\scriptstyle\sigma$}]                (t6) [right=of t5] {};
  \node[dot,label=above:{$\scriptstyle\sigma \circ \eta^{1}$}] (t7) [right=of t6] {};
  \node[dot,label=above:{$\scriptstyle\sigma \circ \eta^{2}$}] (t8) [right=of t7] {};

  \draw[thick] ([xshift=-4mm]sigmaTop.center) -- ([xshift=4mm]t5.center);
  
  \node[dotgray,label=above:{$\scriptstyle\sigma'$}]
    (sigmaPrime) at ($(t3)+(0,-18mm)$) {};
  \node[dotgray,label=above:{$\scriptstyle\sigma' \circ \eta^{1}$}] (b1) [right=of sigmaPrime] {};
  \node[dotgray,label=above:{$\scriptstyle\sigma' \circ \eta^{2}$}] (b2) [right=of b1] {};
  \node[dotgray,label=above:{$\scriptstyle\sigma' \circ \eta^{3}$}] (b3) [right=of b2] {};
  \node[dotgray,label=above:{$\scriptstyle\sigma' \circ \eta^{4}$}] (b4) [right=of b3] {};
  \node[dotgray,label=above:{$\scriptstyle\sigma' \circ \eta^{5}$}] (b5) [right=of b4] {};
  \draw[thick] ([xshift=-4mm]sigmaPrime.center) -- ([xshift=4mm]b5.center);
  \draw[thick,shorten >=12pt,shorten <=12pt]
    (t3) -- node[right=2pt,midway] {$\scriptstyle\eta^{3} = (i,j)$} (sigmaPrime);
  \foreach \top/\bot/\h in {
      t3/sigmaPrime/4mm,
      t4/b1/7mm,
      t5/b2/3mm,
      t6/b3/9mm,
      t7/b4/5mm,
      t8/b5/6mm
  }{
    \draw[fill=black]
      ($(\top)+(-1.5mm,7mm)$) rectangle ($(\top)+(1.5mm,7mm+\h)$);
    \draw[fill=gray!60]
      ($(\bot)+(-1.5mm,-4mm)$) rectangle ($(\bot)+(1.5mm,-4mm-\h)$);
  }

\end{tikzpicture}
\caption{
Aligned case \((I,J)=(i,j)\) with \(\ord(\eta)=2m=6\), so \(m=3\). The two rows show the orbits $\{\sigma \eta^{t}\}_{t=0}^{5}$ and $\{\sigma' \eta^{t}\}_{t=0}^{5}$. Since \(\eta^{3} = \tau_{ij}\), we have \(\sigma' = \sigma \tau_{ij} = \sigma \eta^{3}\), and thus for all $t$ $\sigma' \eta^{t} = \sigma \eta^{t+3}$. In particular, the lower row is just a cyclic shift of the upper row by three positions, so the two orbits induce the same categorical law on indices.
}
\label{fig:orbit-shift}
\end{figure}
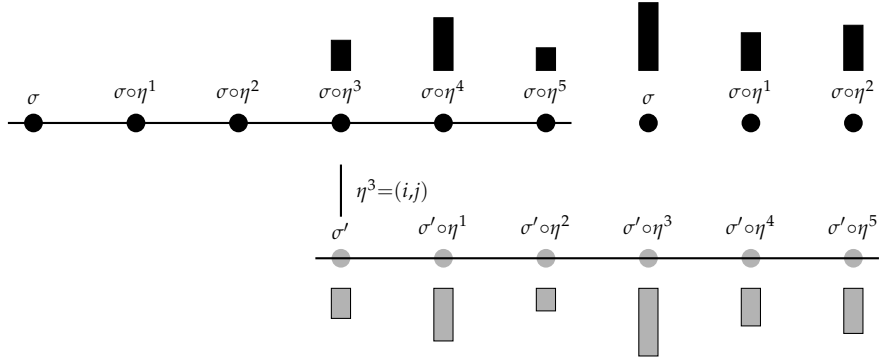

\subsection{Mismatch Cases}
\label{subsec:mismatch}

We now analyze the \emph{mismatch case}, in which the sampled transposition index
$(I,J)$ does not coincide with the edge $(i,j)$ under consideration. In contrast
to the aligned case, the two NURS orbits no longer coincide up to a fixed index
shift. Nevertheless, a key structural property remains: corresponding points
along the two orbits stay at Cayley distance one. This uniform proximity allows
us to control the expected distance after one step via a coupling of the window
index distributions.

\medskip

\begin{lemma}[Unit-distance in the mismatch case]\label{lem:mismatch-dist1}
Fix $1\le i<j\le n$ and let $\tau_{ij}=(i\,j)$. For $\sigma\in S_n$ set
$\sigma' := \sigma \circ \tau_{ij}$. Let $(I,J)\neq(i,j)$ and take any
$\eta\in\Omega_{IJ}$. Then for every $t\in\mathbb{Z}$,
\[
  d_{\mathrm{Cay}}(X_t,Y_t) = 1.
\]
\end{lemma}

\begin{proof}
For each $t$,
\[
  X_t^{-1}Y_t
  = (\sigma\eta^{t})^{-1}\,(\sigma\circ\tau_{ij}\,\eta^{t})
  = \eta^{-t}\tau_{ij}\eta^{t},
\]
which is a transposition (a conjugate of $\tau_{ij}$).
Since $X_t^{-1}Y_t$ is a transposition, $X_t$ and $Y_t$ differ by exactly one
transposition, and therefore $d_{\mathrm{Cay}}(X_t,Y_t)=1$.
\end{proof}

The preceding lemma shows that, although the two orbits may drift apart in index,
they remain uniformly close in state space. We now exploit this fact to bound the
expected Cayley distance after one NURS update in the mismatch case.

\medskip

\begin{theorem}[Expected Cayley distance in the mismatch case]\label{thm:mismatch-Ed}
Consider the shiftable direction kernel $K$ with window length $2m$ and direction law $q$ supported on $\Omega$. Fix $1\le i<j\le n$ and set $\sigma' := \sigma \circ \tau_{ij}$. In the mismatch case $(I,J)\neq(i,j)$, let $U\sim K(\sigma,\cdot)$ and $V\sim K(\sigma',\cdot)$ be coupled via the window-index construction:
\[
  U = X_T = \sigma\eta^{T}, \qquad
  V = Y_{T'} = \sigma'\eta^{T'},
\]
where $T,T'$ are coupled indices on the window $W=\{0,\dots,2m-1\}$. Then there exists a coupling of $(T,T')$ such that
\[
  \EE\bigl[d_{\mathrm{Cay}}(U,V)\bigr]
  \;\le\;
  1 + (D_{\mathrm{cross}}-1)\,\tanh(\beta L_E).
\]

\end{theorem}

\begin{proof}
On the mismatch event $(I,J)\neq(i,j)$, Lemma~\ref{lem:mismatch-dist1} gives
\[
  d_{\mathrm{Cay}}(X_t,Y_t)=1 \qquad\text{for all }t.
\]
Let $p$ and $q$ be the categorical laws of the window index under $K(\sigma,\cdot)$ and $K(\sigma',\cdot)$:
\[
  p(t)\;\propto\; w(X_t),\qquad
  q(t)\;\propto\; w(Y_t),\qquad t\in W.
\]
By Lemma~\ref{lem:mismatch-dist1} we have $d_{\mathrm{Cay}}(X_t,Y_t) =  1$ for all $t$, so in Theorem~\ref{thm:mallows-TV} we may take $k=1$,
\[
  \TV(p,q) \;\le\; \tanh(\beta L_E) . 
\]
Let $(T,T')$ be a maximal coupling of $p$ and $q$, so that
\[
  \mathbb{P}(T\neq T') = \TV(p,q) \le \tanh(\beta L_E).
\]
On the event $\{T\neq T'\}$ we use the trivial bound $d_{\mathrm{Cay}}(U,V)\le D_{\mathrm{cross}}$. Therefore
\begin{align*}
  \EE\bigl[d_{\mathrm{Cay}}(U,V)\bigr]
  &\le
  1\cdot\mathbb{P}(T=T') + D_{\mathrm{cross}}\cdot\mathbb{P}(T\neq T') \\
  &\le
  1\cdot\bigl(1-\tanh(\beta L_E)\bigr)
  + D_{\mathrm{cross}}\,\tanh(\beta L_E) \\
  &= 1 + (D_{\mathrm{cross}}-1)\,\tanh(\beta L_E).
\end{align*}
\end{proof}

\section{Numerical Illustration} \label{sec:numerics}

This section reports numerical experiments for the discrete No--Underrun Sampler
(NURS) applied to Mallows models on the symmetric group \(S_n\). We consider
Mallows distributions defined via the Kendall, \(L^1\), \(L^2\), Hamming, Cayley,
and Ulam distances, and investigate the empirical behavior of the resulting Markov
chains across a range of inverse-temperature regimes. The experiments examine
orbit index distributions, fixed--point statistics, and trace diagnostics, and
include comparisons with standard two--point Metropolis--type updates. The
observed behavior is consistent with the theoretical analysis, including effective
exploration in high--temperature regimes and the emergence of Poisson--type
fixed--point statistics.  Throughout this section we take \(\varepsilon=0.01\)
and a maximal doubling level \(M=7\).

\subsection{Uniform Direction} 

A natural choice of direction law for NURS is to sample the direction uniformly from the symmetric group, $\rho \sim \mathrm{Unif}(S_n)$. With this choice, however, there is no control over how the energy varies along the resulting orbit.  For a typical $\rho$, the sequence $E(\sigma\circ\rho^{k})$ can traverse essentially the full energy range $[0,E_{\max}]$ where $E_{\max} := \max_{\sigma\in S_n} E(\sigma)$. Along such an orbit the weight ratios are
\[
\frac{w_{k+m}}{w_k}
= \exp\!\Big(-\beta\big[E(\sigma\circ\rho^{k+m})-E(\sigma\circ\rho^{k})\big]\Big),
\]
and under uniform directions these ratios can be as extreme as
$e^{\pm \beta E_{\max}}$. When $\beta E_{\max}\gg 1$, the categorical resampling
step concentrates almost all mass on one or two lowest-energy orbit points, so
successive updates return the same state with probability close to one.
Consequently, we expect uniform directions are effective primarily in the ``hot'' regime
$\beta E_{\max} \lesssim 1$. To operate in colder regimes, one must use milder
direction laws that induce smaller energy fluctuations along the orbit.

\medskip

Figure~\ref{fig:index-uniform} displays empirical histograms of the signed index
\(k\) selected by NURS with uniform directions for \(n=200\). Across the three
distance functions shown, the resulting index distributions are approximately
symmetric and triangular. This behavior indicates that the adaptive stopping rule
effectively prevents pathological concentration of mass near \(k=0\), even when
directions are sampled uniformly from the symmetric group.

\medskip

\begin{remark}
If indices along the orbit were sampled uniformly, the signed index would follow
an exactly symmetric triangular distribution centered at $k=0$. The
approximately triangular and symmetric shapes observed in
Figure~\ref{fig:index-uniform} are therefore consistent with only mild
distortions from uniform sampling induced by the weighting and adaptive stopping
mechanisms.
\end{remark}

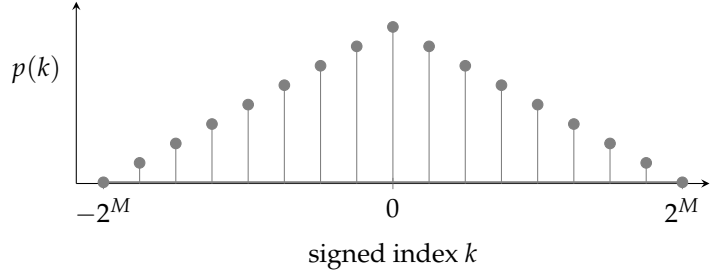
\begin{figure}[t]
\centering
\begin{tikzpicture}
\begin{axis}[
    width=10cm, height=4cm,
    xlabel={signed index $k$},
    ylabel={$p(k)$},
    ylabel style={rotate=-90, anchor=south east},
    xmin=-140, xmax=140,
    ymin=0, ymax=0.009,
    xtick={-128,0,128},
    xticklabels={$-2^M$, $0$, $2^M$},
    ytick=\empty,
    axis x line=bottom,
    axis y line=left,
    grid=none,
    enlargelimits=false,
    clip=false
]

\addplot[thick, gray, domain=-128:128, samples=2]
  { (129 - abs(x)) / (129*129) };

\addplot[ycomb, thin, mark=*, mark options={fill=gray}, draw=gray]
  coordinates {
(-128,1/16641) (-112,17/16641) (-96,33/16641) (-80,49/16641)
(-64,65/16641) (-48,81/16641) (-32,97/16641) (-16,113/16641)
(0,129/16641)
(16,113/16641) (32,97/16641) (48,81/16641) (64,65/16641)
(80,49/16641) (96,33/16641) (112,17/16641) (128,1/16641)
};

\end{axis}
\end{tikzpicture}
\caption{If the orbit lengths are of fixed length $2^M$, and the indices are sampled uniformly along the constructed orbit (equivalently, all orbit weights are equal), then
the signed index $k$ has the symmetric triangular pmf
$p(k)=\frac{2^M+1-|k|}{(2^M+1)^2}$ for $|k|\le 2^M$.
In our experiments $M=7$, hence $2^M=128$.}
\label{fig:triangular-pmf}
\end{figure}

In the hot regime $\beta \approx 1/E_{\max}$, the Mallows distribution is close to
the uniform measure on $S_n$, under which the number of fixed points is known to
converge in distribution to $\mathrm{Poisson}(1)$ as $n\to\infty$; see, e.g.,~\cite{Mukherjee2016FixedPoints}.
Using $2{,}000$ burn-in iterations followed by $8{,}000$ retained samples, the
empirical fixed-point histograms in Figure~\ref{fig:fixed-unif} closely match the
$\mathrm{Poisson}(1)$ pmf, with only minor discrepancies attributable to
finite-sample variability.

\begin{figure}[H]
\centering

\begin{subfigure}{0.32\linewidth}
  \includegraphics[width=\linewidth]{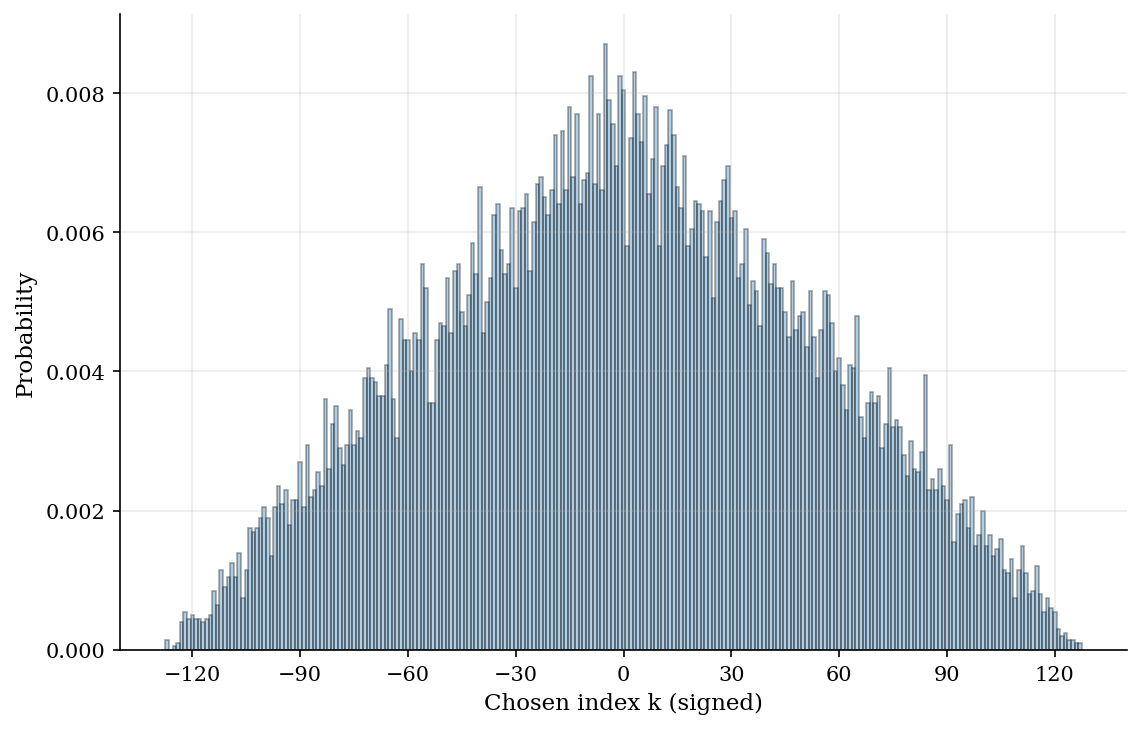}
  \caption{Ulam, $\beta \sim n^{-1}$.}
  \label{fig:index-ulam}
\end{subfigure}\hfill
\begin{subfigure}{0.32\linewidth}
  \includegraphics[width=\linewidth]{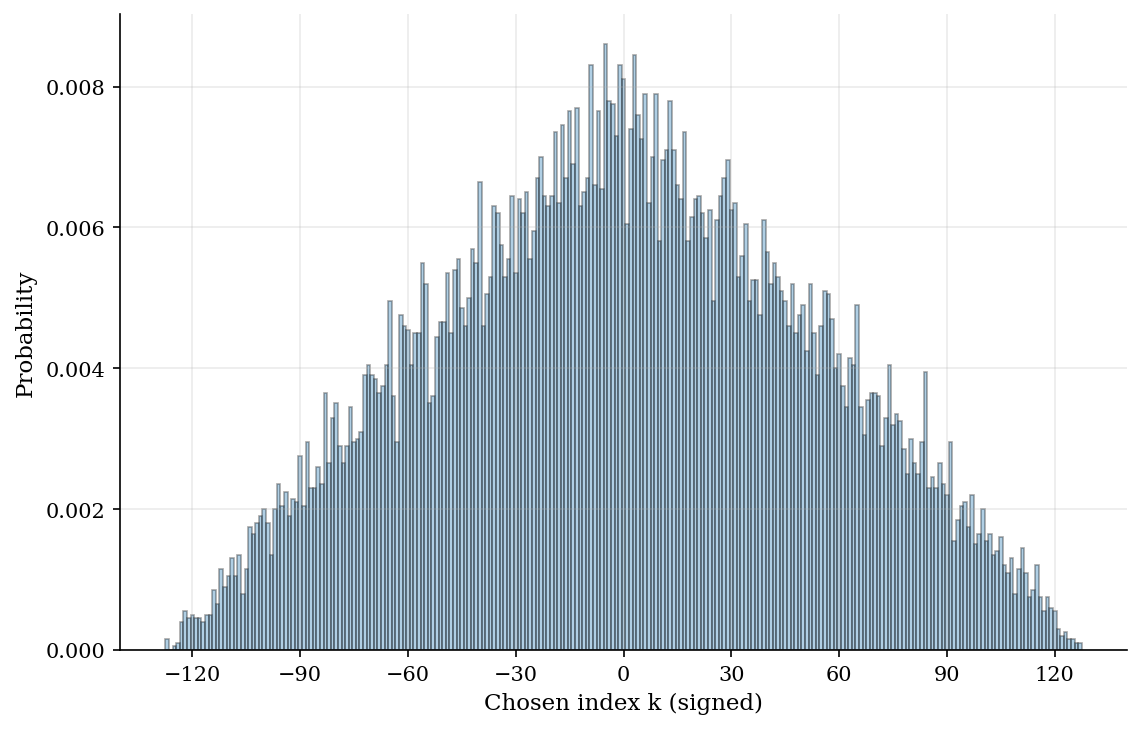}
  \caption{$L^1$, $\beta \sim n^{-2}$.}
  \label{fig:index-l1}
\end{subfigure}\hfill
\begin{subfigure}{0.32\linewidth}
  \includegraphics[width=\linewidth]{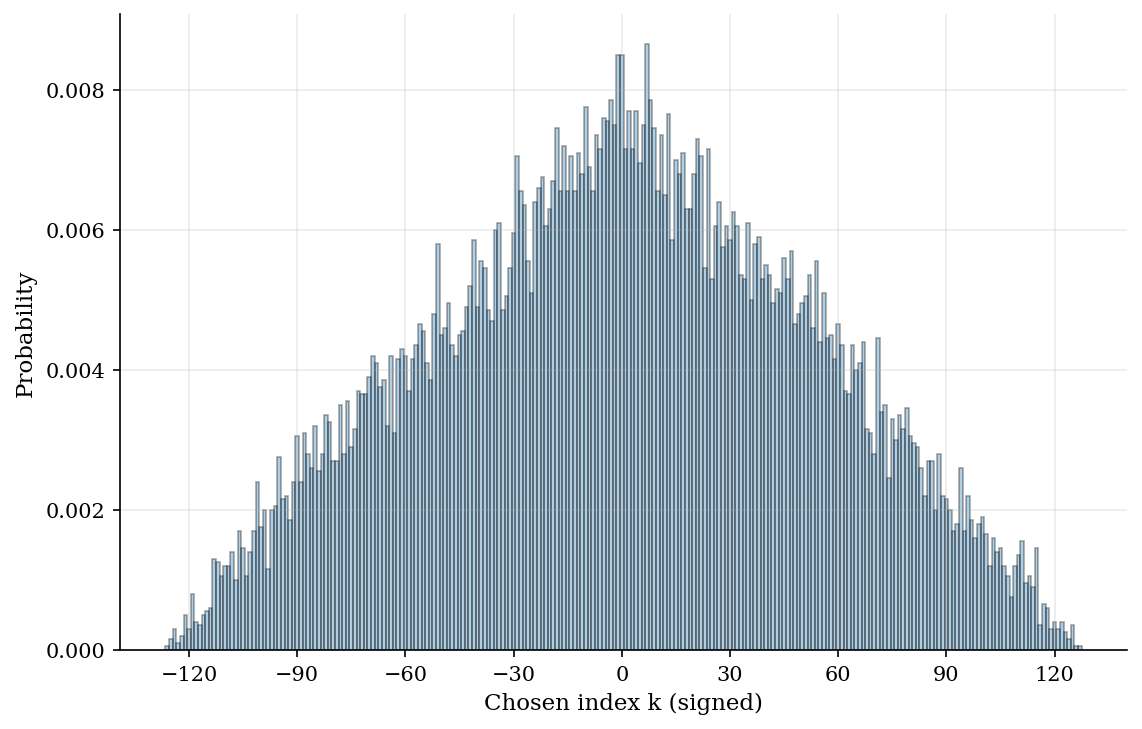}
  \caption{$L^2$, $\beta \sim n^{-3}$.}
  \label{fig:index-l2}
\end{subfigure}
\caption{Signed index $k$ for NURS with Uniform directions with scaling of $\beta=1/E_{\max}$.}
\label{fig:index-uniform}
\end{figure}

\begin{figure}[H]
\begin{subfigure}{0.32\linewidth}
  \includegraphics[width=\linewidth]{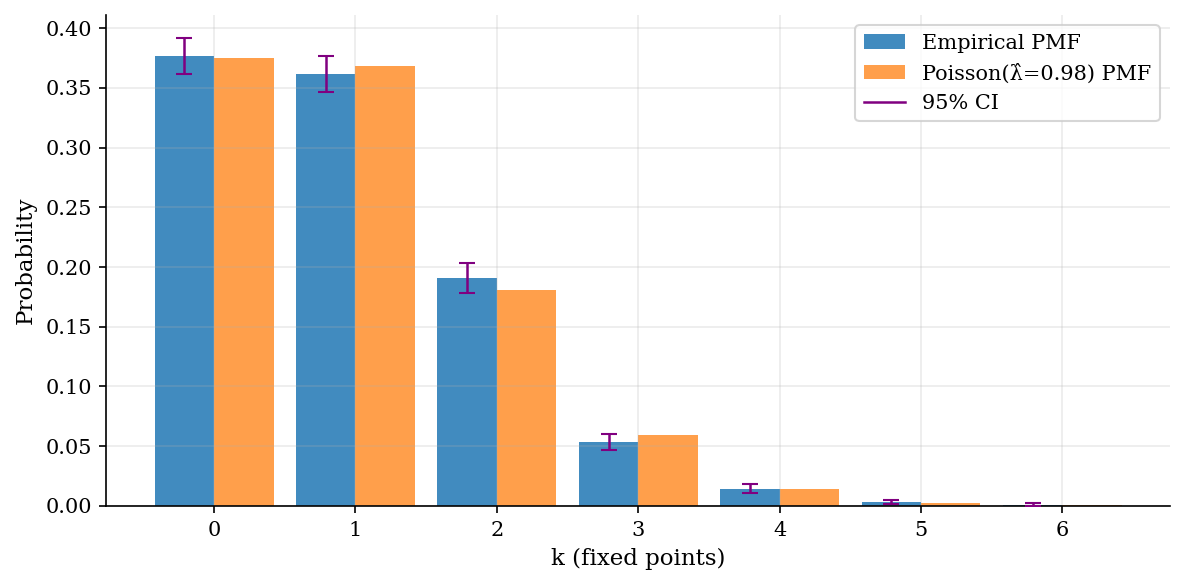}
  \caption{Ulam, $\beta=n^{-1}$.}
  \label{fig:fixed-unif-kendall}
\end{subfigure}\hfill
\begin{subfigure}{0.32\linewidth}
  \includegraphics[width=\linewidth]{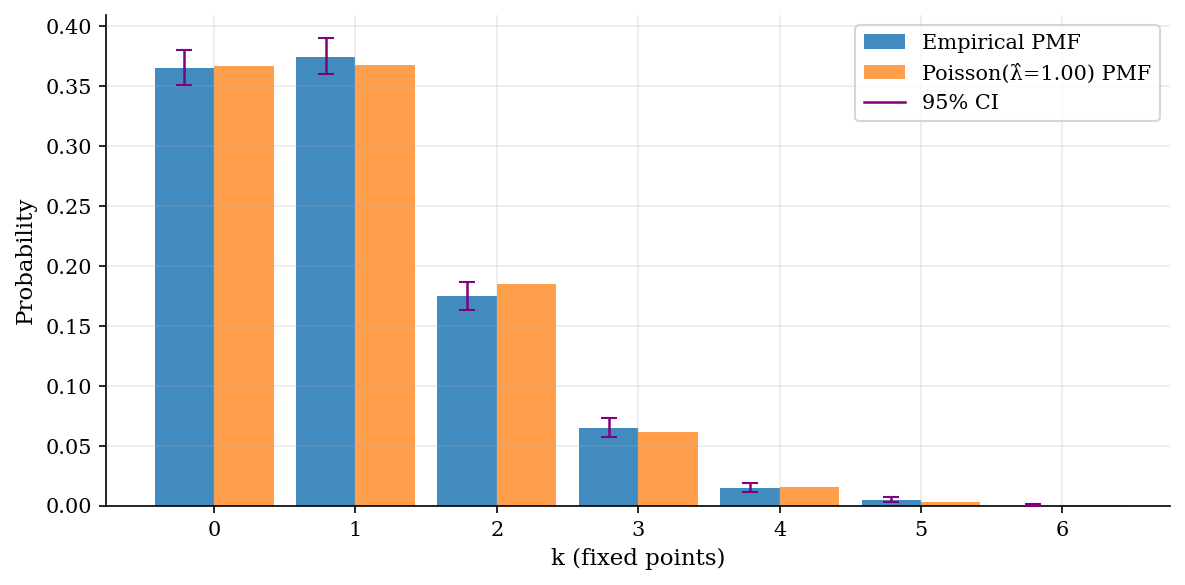}
  \caption{$L^1$, $\beta \sim n^{-2}$.}
  \label{fig:fixed-unif-L1}
\end{subfigure}\hfill
\begin{subfigure}{0.32\linewidth}
  \includegraphics[width=\linewidth]{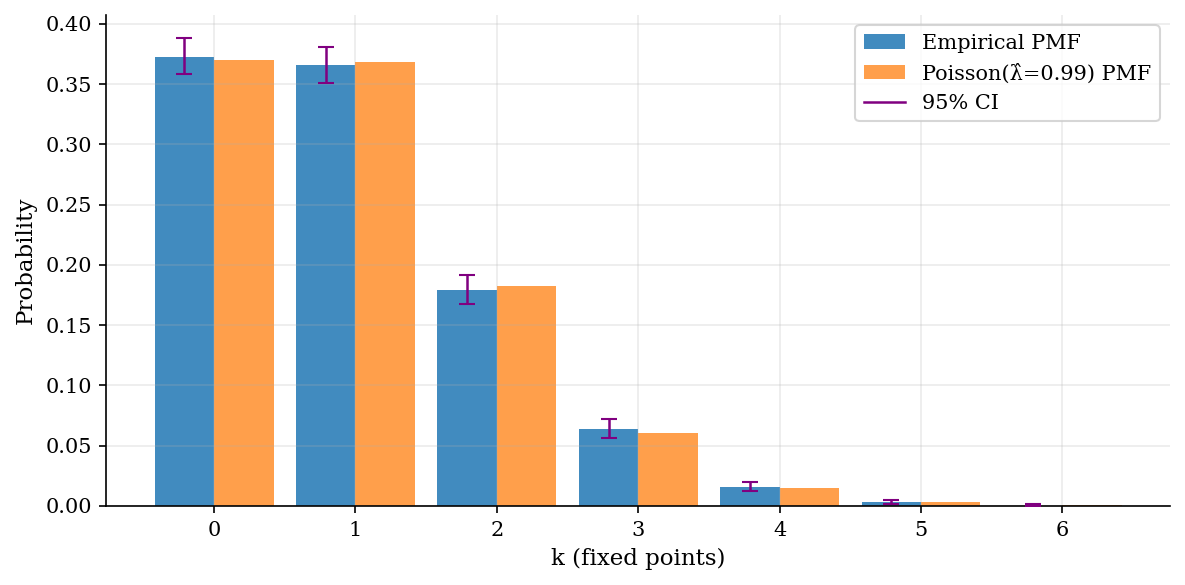}
  \caption{$L^2$,$\beta \sim n^{-3}$.}
  \label{fig:fixed-unif-L2}
\end{subfigure}
\caption{Fixed point distribution with scaling of $\beta=1/E_{\max}$.}
\label{fig:fixed-unif}
\end{figure}

\subsection{Block--shuffle direction}

Uniformly random directions on $S_n$ typically induce large, uncontrolled energy
fluctuations along NURS orbits. The block--shuffle construction provides a simple
family of \emph{milder} directions that interpolate between local moves and
global shuffles by tuning a single block-size parameter $B$.

\medskip
Fix a block size $B\in\{2,\dots,n\}$. For an offset $U\in\{0,\dots,B-1\}$, define
the cyclic rotation
\[
  \rot{U}(t) \;:=\; 1 + \bigl((t+U-1)\bmod n\bigr),
  \qquad t=1,\dots,n.
\]
This rotation randomizes the alignment of blocks relative to the labeling
$\{1,\dots,n\}$.

\medskip

\begin{definition}[Block--shuffle direction law $q_B$]
\label{def:block-shuffle-law}
A random permutation $\rho\in S_n$ with law $q_B$ is generated as follows:
\begin{enumerate}
  \item Sample an offset $U\sim\Unif\{0,1,\dots,B-1\}$.
  \item Work in the rotated coordinates $t'=\rot{U}^{-1}(t)$ and partition
        $\{1,\dots,n\}$ into consecutive blocks of size $B$
        (except possibly the final block).
  \item On each block, independently sample a uniform random permutation.
        Let $\Pi$ denote the resulting blockwise permutation in the rotated
        coordinates.
  \item Define the final permutation on the original labels by
        \[
          \rho(t)
          \;:=\;
          \rot{U}\bigl(\Pi(\rot{U}^{-1}(t))\bigr),
          \qquad t=1,\dots,n.
        \]
\end{enumerate}
\end{definition}

\begin{remark}[Basic properties of the block--shuffle direction law]
The law $q_B$ is state independent. For each $i=1,\dots,n-1$, there exists an
offset $U$ that places the adjacent labels $i$ and $i+1$ in the same block.
Conditional on this choice of $U$, selecting the within-block permutation that
swaps these two labels and taking the identity permutation on all other blocks
produces the adjacent transposition $\rho=(i\,i{+}1)$ with positive probability.
Consequently, the support of $q_B$ contains all adjacent transpositions and thus
generates the symmetric group $S_n$. Since $q_B$ also assigns positive probability
to the identity permutation, the induced Markov chain is irreducible and
aperiodic, and hence ergodic on $S_n$.
\end{remark}
\medskip
\begin{example}[Block structure on $S_{10}$ with $B=4$]
Let $n=10$ and $B=4$. Suppose $U=1$, so $\rot{1}(t)=t+1\pmod{10}$. In rotated
coordinates the blocks are
\[
  \{1,2,3,4\},\quad \{5,6,7,8\},\quad \{9,10\},
\]
which correspond in the original labels to
\[
  \{2,3,4,5\},\quad \{6,7,8,9\},\quad \{10,1\}.
\]
Sampling independent uniform permutations on these blocks and composing them
produces the block--shuffle permutation $\rho$. Away from the wrapped block
$\{10,1\}$, each element moves by at most $B-1$ positions.
\end{example}

\medskip

\begin{remark}[Why block--shuffle helps]
Let $\rho\sim q_B$ be a random permutation drawn from the block--shuffle direction
law (Definition~\ref{def:block-shuffle-law}). By construction, all but at most $B$ labels satisfy
\[
|\rho(k)-k|\le B-1
\]
the exception being the single block that wraps around the boundary between $1$
and $n$.  Thus, each label moves by at most $O(B)$ positions, and large global
rearrangements are excluded.

As a consequence, along any NURS orbit $\{\sigma\circ\rho^k\}_k$ generated using
$\rho\sim q_B$, the one--step energy increment
\[
\Delta := E(\sigma\circ\rho)-E(\sigma)
\]
grows in a controlled way with $n$, with constants depending on the block size
$B$. In particular, catastrophic jumps of order $E_{\max}$ (which are typical for
uniform directions) are ruled out.

To see the resulting scaling, it suffices to take $\sigma=\sigma_0$, the identity
permutation. In this case,
\[
\begin{aligned}
\text{$L^2$:}\quad
&\Delta_{L^2}
= E_{L^2}(\rho)
= \sum_k \bigl(\rho(k)-k\bigr)^2
\;\lesssim\; n^2 B,\\[3pt]
\text{$L^1$:}\quad
&\Delta_{L^1}
= E_{L^1}(\rho)
= \sum_k \bigl|\rho(k)-k\bigr|
\;\lesssim\; n B,\\[3pt]
\text{Kendall:}\quad
&\Delta_{\tau}
= E_{\tau}(\rho)
\;\lesssim\; n B,\\[3pt]
\text{Cayley:}\quad
&\Delta_{\mathrm{Cay}}
= E_{\mathrm{Cay}}(\rho)
= n-\#\mathrm{cycles}(\rho)
\;\lesssim\; n,\\[3pt]
\text{Hamming:}\quad
&\Delta_{\mathrm{Ham}}
= E_{\mathrm{Ham}}(\rho)
\;\le\; n,\\[3pt]
\text{Ulam:}\quad
&\Delta_{\mathrm{Ulam}}
= E_{\mathrm{Ulam}}(\rho)
= n-\mathrm{LIS}(\rho)
\;\lesssim\; n.
\end{aligned}
\]

Accordingly, the weight ratios along the orbit satisfy
\[
\frac{w_{k+1}}{w_k}=\exp(-\beta\,\Delta),
\qquad\text{so that}\qquad
\beta|\Delta|\ \text{can be tuned to remain $O(1)$.}
\]
In high--temperature regimes, where $\beta$ scales inversely with the natural
growth rate of $\Delta$, this ensures that $\beta\Delta=O(1)$ and prevents the
extreme $e^{\pm\beta E_{\max}}$ fluctuations characteristic of uniform directions
$\mathrm{Unif}(S_n)$. The symmetric, triangular index distributions observed in
Figure~\ref{fig:index-w} are consistent with this controlled energy variation.

By contrast, for Mallows models based on Cayley, Hamming, or Ulam distances at
fixed inverse temperature $\beta=1$, maintaining $\beta\Delta=O(1)$ requires
$\Delta=O(1)$, which forces only $O(1)$ labels to move per step. In this cold
regime, block--shuffle directions are therefore no longer appropriate, and we
instead employ the local direction law (Definition~\ref{def:local-direction}).
\end{remark}

\begin{figure}[H]
\begin{subfigure}{0.32\linewidth}
  \includegraphics[width=\linewidth]{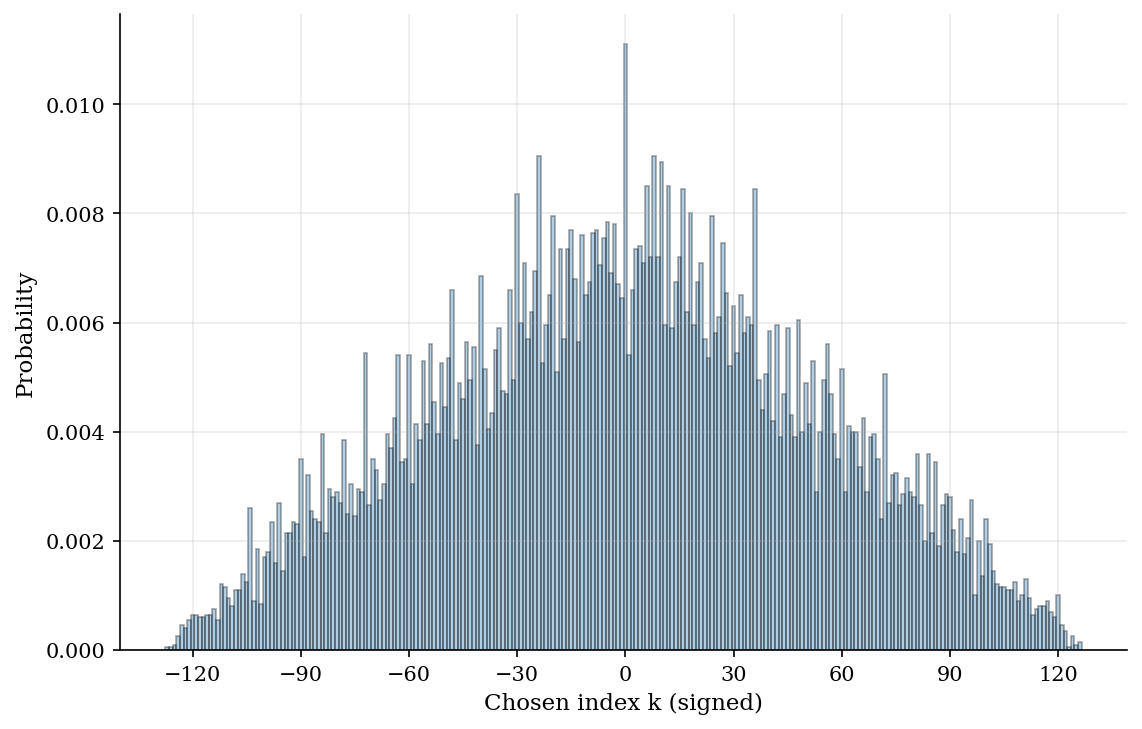}
  \caption{Kendall, $\beta=n^{-1}$, $B=9$.}
  \label{fig:index-inv-w}
\end{subfigure}\hfill
\begin{subfigure}{0.32\linewidth}
  \includegraphics[width=\linewidth]{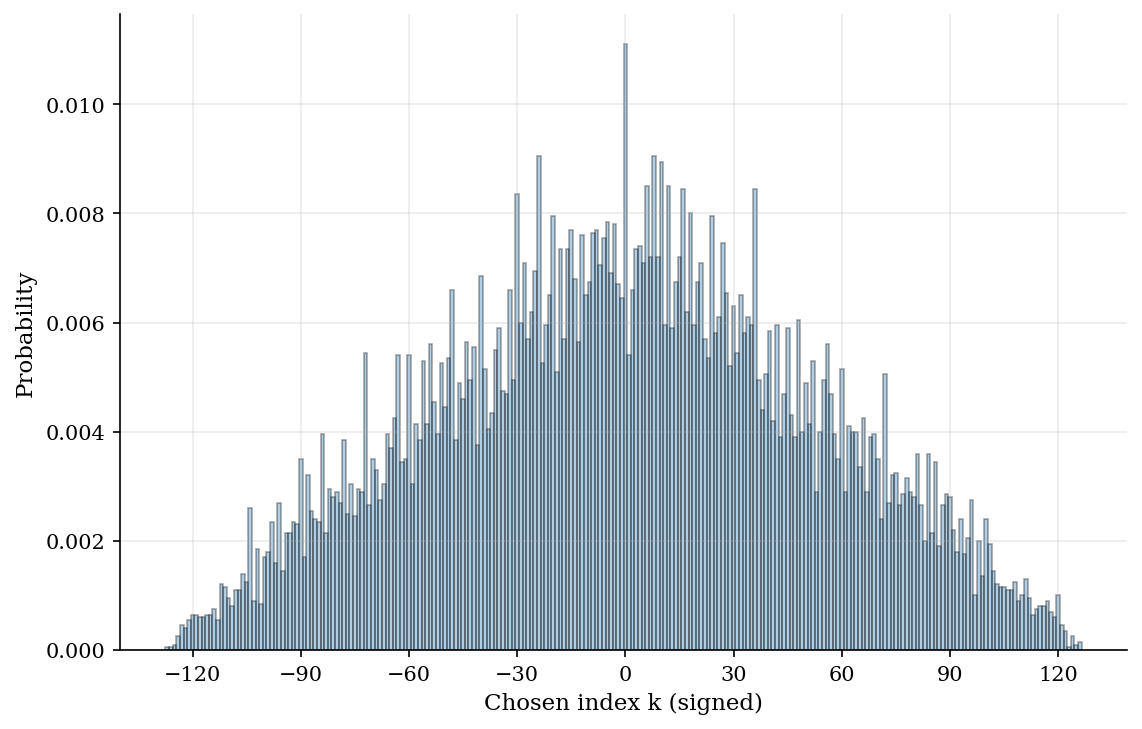}
  \caption{$L^1$, $\beta \sim n^{-1}$, $B=9$.}
  \label{fig:index-L1-w}
\end{subfigure}\hfill
\begin{subfigure}{0.32\linewidth}
  \includegraphics[width=\linewidth]{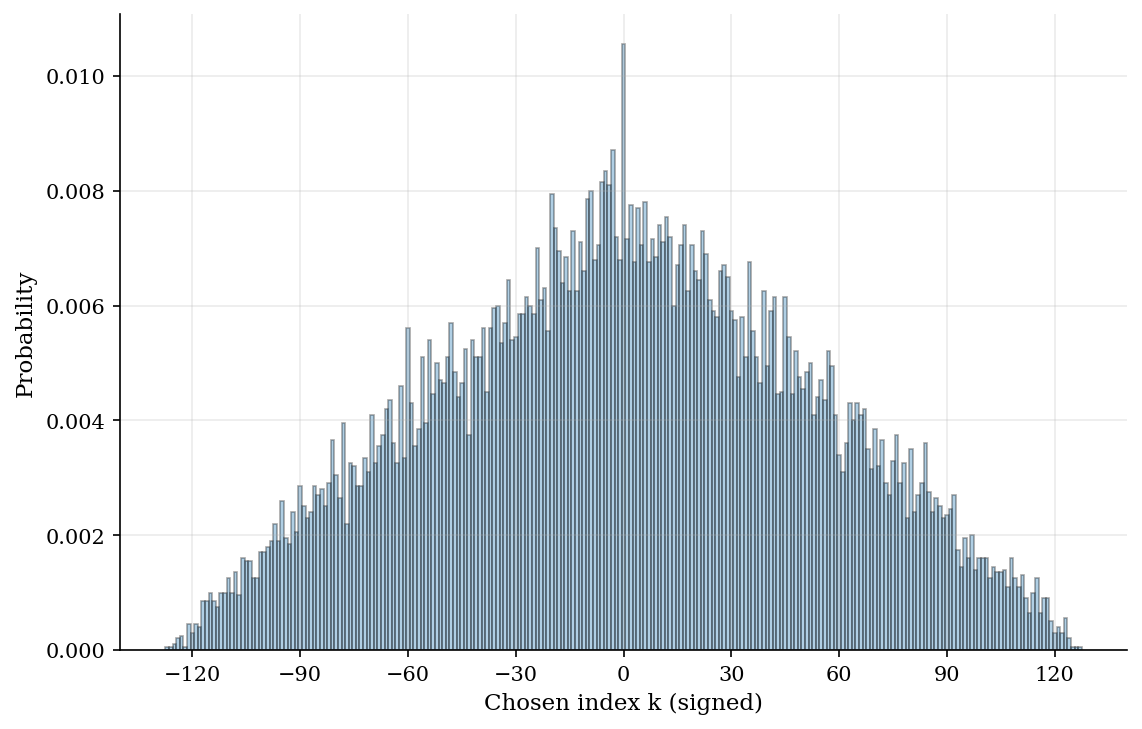}
  \caption{$L^2$,$\beta \sim n^{-2}$, $B=15$.}
  \label{fig:index-L2-w}
\end{subfigure}
\caption{Signed index $k$ for NURS with block–shuffle directions.}
\label{fig:index-w}
\end{figure}

\subsection{Local Direction}

\begin{definition}[Local direction law $q_\ell$]
\label{def:local-direction}
Fix an integer $\ell\in\{2,\dots,n\}$.  
A random direction $\rho\in S_n$ with law $q_\ell$ is generated as follows:
\begin{enumerate}
  \item Sample a starting index $i$ uniformly from $[n]:=\{1,\dots,n\}$.
  \item For $r=0,1,\dots,\ell-1$, define
  \[
    i_r \;:=\; 1 + \bigl((i-1+r)\bmod n\bigr)\ \in\ [n].
  \]
  \item Let $\rho$ be the $\ell$--cycle on these indices,
  \[
    \rho \;:=\; (i_0\ i_1\ \cdots\ i_{\ell-1}),
  \]
  i.e.\ $\rho(i_r)=i_{r+1}$ for $r=0,\dots,\ell-2$, $\rho(i_{\ell-1})=i_0$, and
  $\rho(k)=k$ for all $k\notin\{i_0,\dots,i_{\ell-1}\}$.
\end{enumerate}
We write $\rho\sim q_\ell$. The law $q_\ell$ is state independent (i.e.,
$q(\cdot\mid\sigma)\equiv q_\ell$ for all $\sigma\in S_n$) and is orbit--equivariant
in the sense of Definition~\ref{def:direction-law}.
\end{definition}

\medskip

\begin{example}[Local direction on $S_{10}$ with $\ell=3$]
Let $n=10$ and $\ell=3$. If the sampled starting index is $i=9$, then the
consecutive indices modulo $10$ are $9,10,1$, and the resulting direction is the
$3$--cycle
\[
\rho=(9\ 10\ 1).
\]
\end{example}

\subsection{\texorpdfstring{$L^1$ and $L^2$ Mallows}{L1 and L2 Mallows}}

For Mallows models based on the $L^1$ and $L^2$ distances, we employ the
\emph{block--shuffle direction law} $q_B$ (Definition~\ref{def:block-shuffle-law}).
These choices correspond to high--temperature scalings
$\beta \asymp n^{-1}$ for $L^1$ and $\beta \asymp n^{-2}$ for $L^2$, under which
the energy increments induced by block--shuffle directions remain controlled.

Figures~\ref{fig:fixed-L1} and~\ref{fig:fixed-L2} display empirical histograms of the
number of fixed points produced by NURS with block–shuffle directions for the
$L^1$ and $L^2$ Mallows models. In both settings, the fixed–point count is well
approximated by a Poisson distribution with parameter $\lambda(\beta)$, which
depends on $\beta$ and satisfies $\lambda(\beta)>1$. All results are based on
$2{,}000$ burn–in iterations followed by $8{,}000$ retained samples. 

\medskip

\begin{remark}
This behavior is consistent with known Poisson limit theorems for fixed points
in Mallows-type permutation models in weak-interaction regimes. In particular,
Mukherjee~\cite{Mukherjee2016PermLimits,Mukherjee2016FixedPoints} showed that when
$\beta$ scales with $n$ as $\beta\sim c/n$ for the $L^1$ model (respectively,
$\beta\sim c/n^2$ for the $L^2$ model), the associated permutations converge in
the permutation–limit sense, and under this convergence, the number of fixed
points converges in distribution to a Poisson random variable whose mean is
determined by the limiting measure. While these results apply to asymptotic
scaling limits rather than to fixed-$n$ models at arbitrary $\beta$, they
provide theoretical support for the Poisson-like behavior observed in our
numerical experiments.
\end{remark}

\begin{figure}[H]
  \centering
  \begin{subfigure}{0.48\textwidth}
    \centering
    \includegraphics[width=\linewidth]{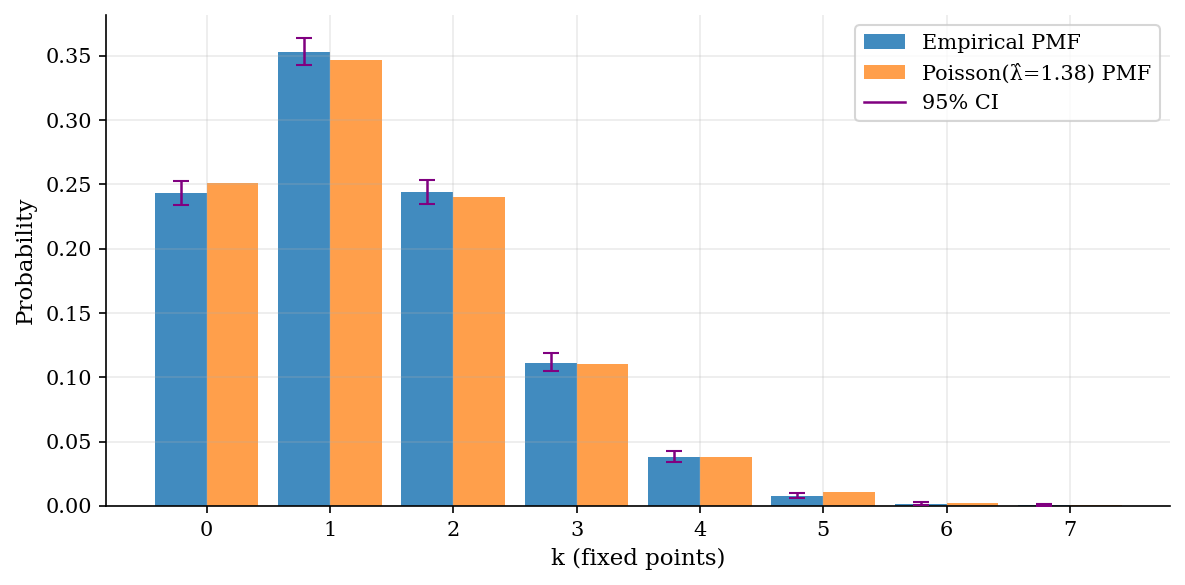}
    \caption{$\beta = n^{-1}$, $B=9$.}
    \label{fig:fixed-L1-a}
  \end{subfigure}\hfill
  \begin{subfigure}{0.48\textwidth}
    \centering
    \includegraphics[width=\linewidth]{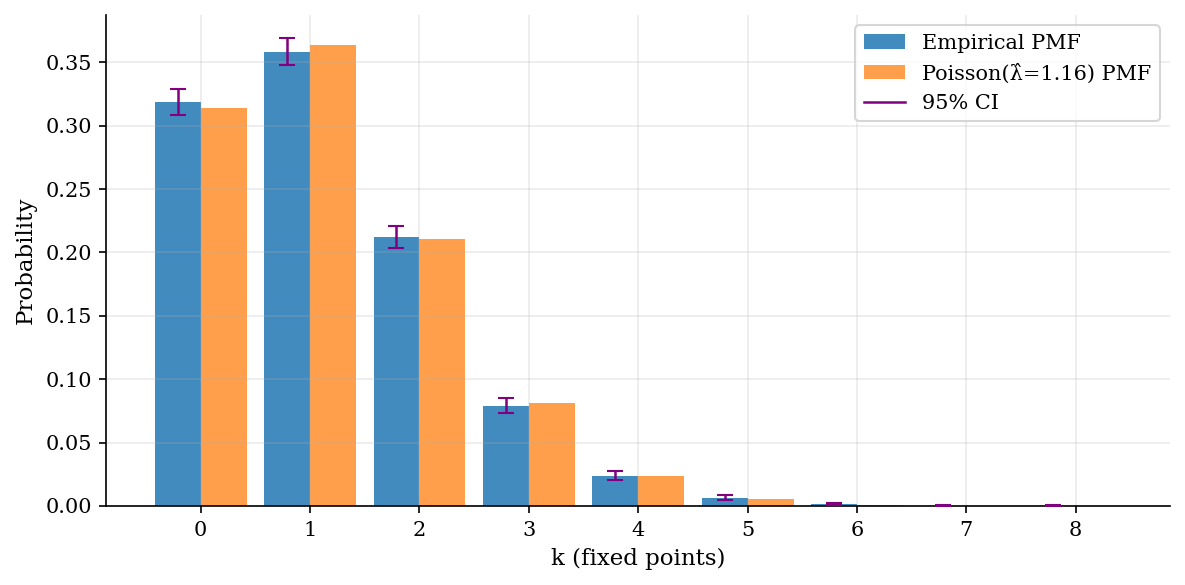}
    \caption{$\beta = 0.5\,n^{-1}$, $B=18$.}
    \label{fig:fixed-L1-b}
  \end{subfigure}
  \caption{Fixed Points of NURS under $L^1$.}
  \label{fig:fixed-L1}
\end{figure}

\begin{figure}[H]
  \centering
  \begin{subfigure}{0.48\textwidth}
    \centering
    \includegraphics[width=\linewidth]{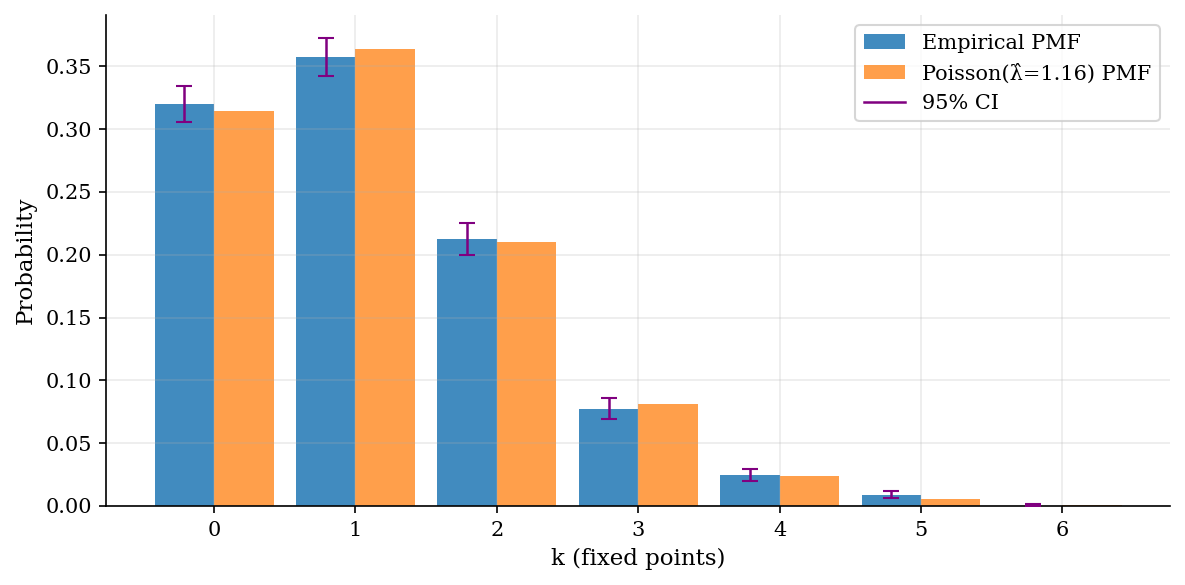}
    \caption{$\beta = n^{-2}$, $B=15$.}
    \label{fig:fixed-L2-1.0}
  \end{subfigure}\hfill
  \begin{subfigure}{0.48\textwidth}
    \centering
    \includegraphics[width=\linewidth]{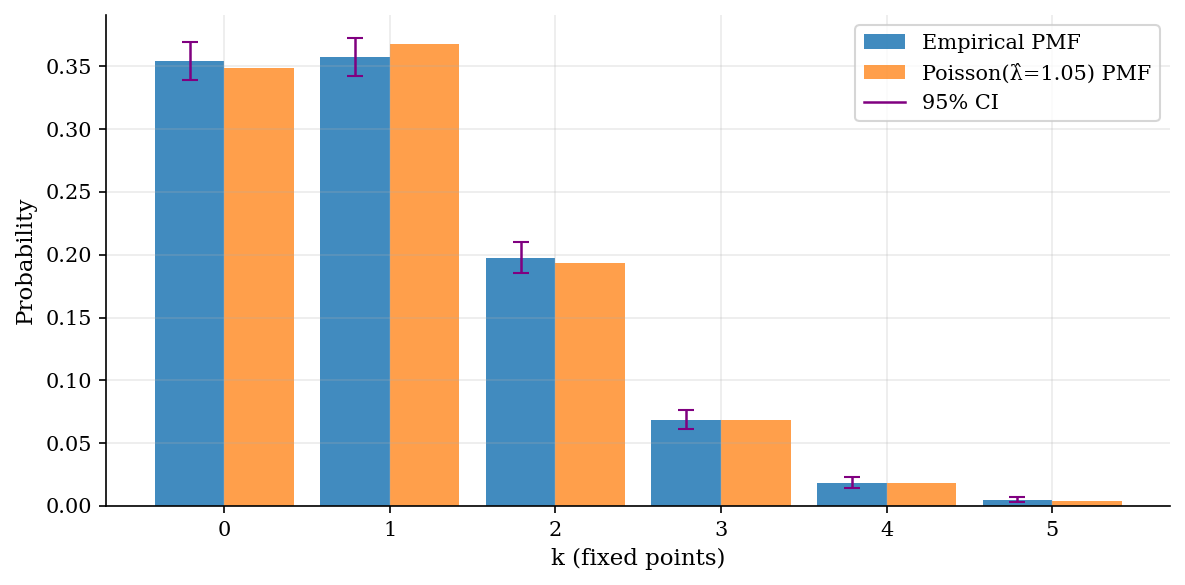}
    \caption{$\beta = 0.5\,n^{-2}$, $B=30$.}
    \label{fig:fixed-L2-b}
  \end{subfigure}
  \caption{Fixed Points of NURS under $L^2$.}
  \label{fig:fixed-L2}
\end{figure}

\subsection{Hamming Mallows}

For Mallows models based on the Hamming distance at inverse temperature
$\beta=O(1)$, we employ the \emph{local direction law} introduced in
Definition~\ref{def:local-direction}. Owing to the strong locality of these
moves, achieving exploration comparable to that obtained with block--shuffle
directions requires approximately $n^2$ more iterations.

For a permutation $\sigma\in S_n$, let $C_1(\sigma)$ denote the number of fixed
points of $\sigma$. When the reference permutation is the identity
$\sigma_0$, the Hamming distance satisfies
$d_{\mathrm{Ham}}(\sigma,\sigma_0)=n-C_1(\sigma)$, and the corresponding Mallows
distribution takes the form
\[
P_\beta(\sigma)\ \propto\ \exp\!\bigl[-\beta\,(n-C_1(\sigma))\bigr]
= e^{-\beta n}\,q^{C_1(\sigma)}, \qquad q := e^{\beta}.
\]
Thus, the Hamming--Mallows model coincides with the fixed--point--biased measure
$\mathbb{P}^{(q)}_n$ defined by
$\mathbb{P}^{(q)}_n(\sigma)\propto q^{\mathrm{fp}(\sigma)}$.

By \cite[Theorem~1]{ChelikavadaPanzo2023a}, under the fixed--point--biased measure
$\mathbb{P}^{(q)}_n$ with $q>0$ fixed, the number of fixed points converges in
distribution to a Poisson random variable with mean $q$:
\[
C_1 \ \xRightarrow[n\to\infty]{}\ \mathrm{Poisson}(q).
\]
Since the Hamming--Mallows model with parameter $\beta$ corresponds to the choice
$q=e^{\beta}$, it follows that
\[
C_1 \ \xRightarrow[n\to\infty]{}\ \mathrm{Poisson}\!\big(e^{\beta}\big).
\]

Figure~\ref{fig:fixed-Hamming} shows the empirical histogram of fixed points for
$n=1000$ together with the $\mathrm{Poisson}(e^{\beta})$ pmf; the close agreement
is consistent with this asymptotic limit.

\begin{figure}[H]
  \centering
  \begin{subfigure}{0.48\textwidth}
    \centering
    \includegraphics[width=\linewidth]{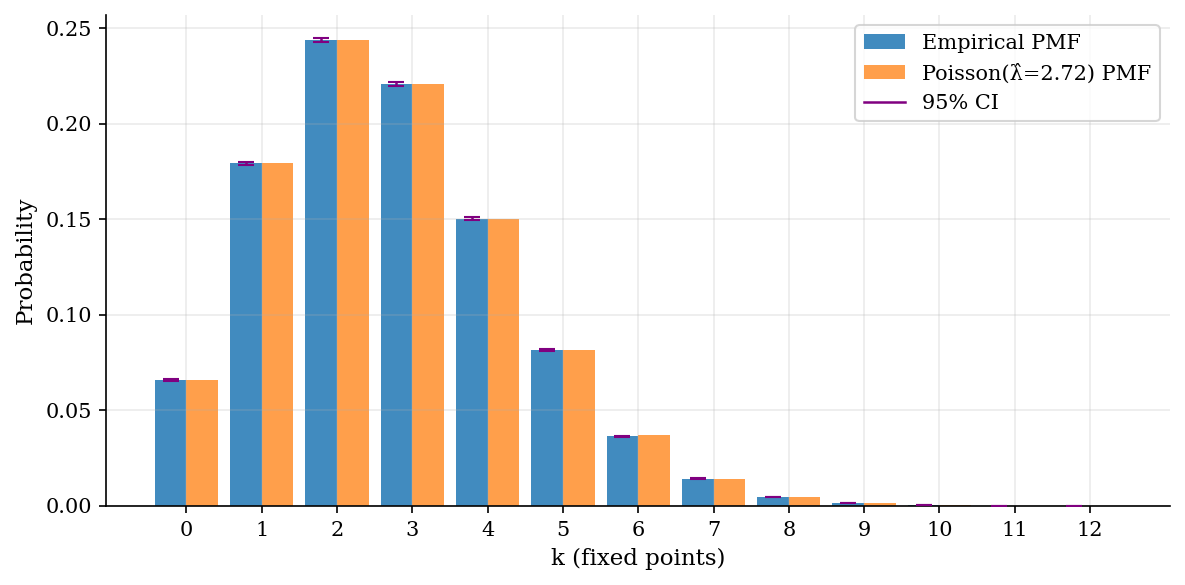}
    \caption{$\beta = 1.0$, $\ell=7$.}
    \label{fig:fixed-Hamming-a}
  \end{subfigure}\hfill
  \begin{subfigure}{0.48\textwidth}
    \centering
    \includegraphics[width=\linewidth]{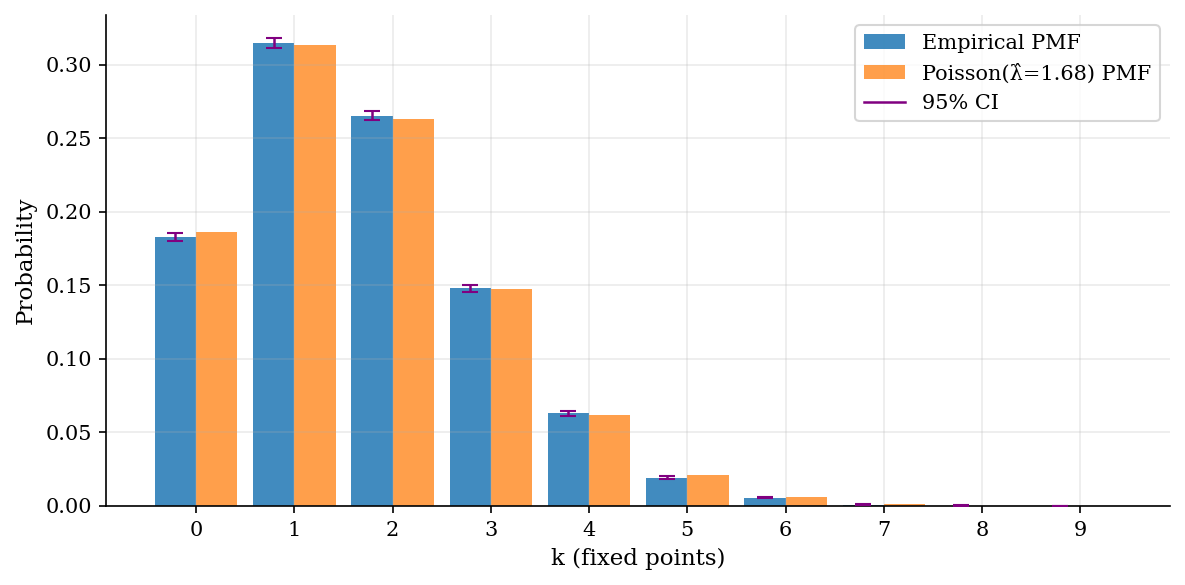}
    \caption{$\beta = 0.5$, $\ell=7$.}
    \label{fig:fixed-Hamming-b}
  \end{subfigure}
  \caption{Fixed Points of NURS under Hamming.}
  \label{fig:fixed-Hamming}
\end{figure}

\subsection{Cayley Mallows}

For the Cayley distance, recall that
\[
d_{\mathrm{Cay}}(\sigma,\sigma_0)= n-\cyc(\sigma),
\]
where $\cyc(\sigma)$ denotes the number of cycles of $\sigma$. The corresponding
Mallows distribution therefore takes the form
\[
P_\beta(\sigma)\ \propto\ \exp\!\bigl[-\beta\,(n-\cyc(\sigma))\bigr]
= e^{-\beta n}\,e^{\beta\,\cyc(\sigma)}.
\]
Equivalently, $P_\beta$ coincides with the Ewens distribution on $S_n$ with
parameter
\[
\varphi := e^{\beta}>0,
\qquad
P_\beta(\sigma)\ \propto\ \varphi^{\cyc(\sigma)};
\]
see, for example, \cite{DeODonnellServedio2021}.

A classical Poisson–process limit for Ewens permutations states that, as
$n\to\infty$,
\[
\bigl(C_1(n),C_2(n),\dots\bigr)
\ \Rightarrow\
\bigl(Z_1,Z_2,\dots\bigr),
\qquad
Z_j \stackrel{\mathrm{ind}}{\sim} \mathrm{Poisson}(\varphi/j),
\]
where $C_j(n)$ denotes the number of cycles of length $j$; see
\cite[Theorem~1]{ArratiaBarbourTavare1992}. In particular, the number of fixed
points $C_1$ converges in distribution to
\[
C_1 \ \Rightarrow\ \mathrm{Poisson}(\varphi)=\mathrm{Poisson}(e^{\beta}).
\]
Thus $C_1$ has asymptotic mean and variance $e^{\beta}$, which explains the close
agreement between the empirical fixed--point histogram and the
$\mathrm{Poisson}(e^{\beta})$ overlay observed in Figure~\ref{fig:fixed-Cayley}.

\begin{figure}[H]
  \centering
  \begin{subfigure}{0.48\textwidth}
    \centering
    \includegraphics[width=\linewidth]{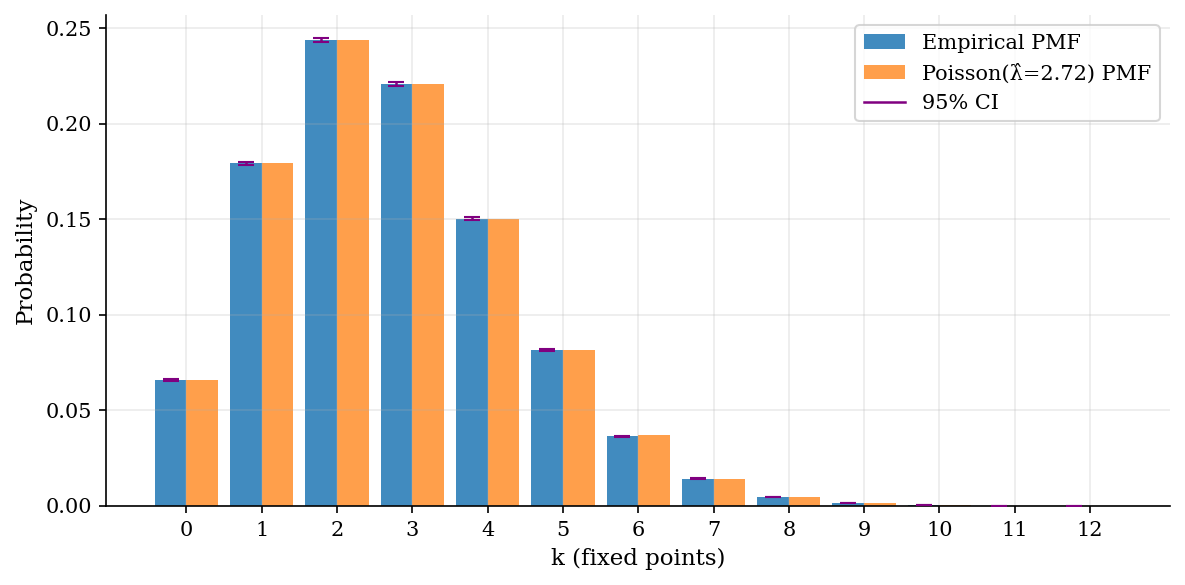}
    \caption{$\beta = 1.0$, $\ell=7$.}
    \label{fig:fixed-Cayley-a}
  \end{subfigure}\hfill
  \begin{subfigure}{0.48\textwidth}
    \centering
    \includegraphics[width=\linewidth]{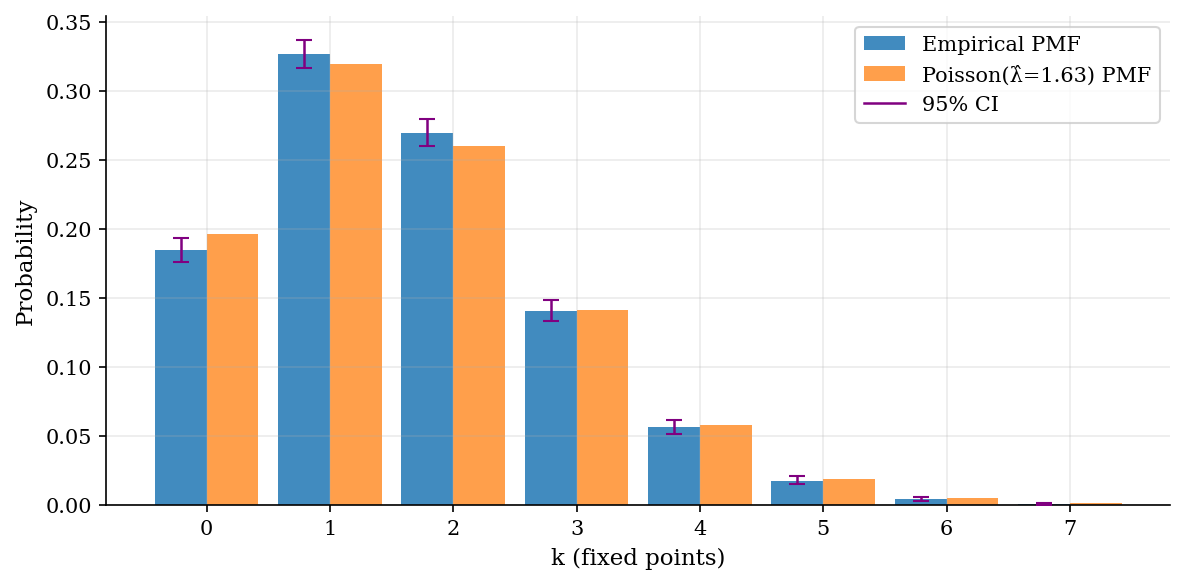}
    \caption{$\beta = 0.5$, $\ell=7$.}
    \label{fig:fixed-Cayley-b}
  \end{subfigure}
  \caption{NURS under Cayley.}
  \label{fig:fixed-Cayley}
\end{figure}

\subsection{Comparison of NURS and Barker}
If we cap the doubling depth at $M=1$, the NURS construction explores a
two--point orbit
\[
\mathcal O=\{\sigma,\sigma'\}, \qquad \sigma'=\sigma\circ\rho.
\]
In this case, the NURS selection step draws the next state from $\mathcal O$ via
categorical resampling with probabilities
\[
K_{\mathrm{NURS}}(\sigma,\sigma')
=\frac{w(\sigma')}{w(\sigma)+w(\sigma')},
\qquad
K_{\mathrm{NURS}}(\sigma,\sigma)
=\frac{w(\sigma)}{w(\sigma)+w(\sigma')},
\]
where \(w(\sigma)\propto e^{-\beta E(\sigma)}\) denotes the unnormalized Mallows
weight.  For a \emph{symmetric} proposal between
$\sigma$ and $\sigma'$, Barker’s acceptance probability \cite{Barker1965,Peskun1973} is
\[
\alpha_{\mathrm{Barker}}(\sigma\to\sigma')
=\frac{w(\sigma')}{w(\sigma)+w(\sigma')}.
\]
The resulting Barker transition kernel is therefore given by
\[
K_{\mathrm{Barker}}(\sigma,\sigma')
=\frac{w(\sigma')}{w(\sigma)+w(\sigma')},
\qquad
K_{\mathrm{Barker}}(\sigma,\sigma)
=\frac{w(\sigma)}{w(\sigma)+w(\sigma')}.
\]
Thus, when the explored orbit consists of two points, NURS’s 
categorical sampling coincides exactly with Barker’s transition rule.

We compare trace plots of three summary statistics: (i) the number of fixed
points, (ii) the cycle length of label~1, and (iii) the length of the longest
increasing subsequence (LIS), for NURS and Barker across three different distance
functions, with inverse temperatures satisfying $\beta=O(n/E_{\max})$; see
Figures~\ref{fig:trace-kendall}, \ref{fig:trace-L1}, and~\ref{fig:trace-L2}. In
all cases (NURS on the left, Barker on the right), the traces exhibit markedly
faster mixing for NURS, corresponding to substantially smaller lag
autocorrelations and larger effective sample sizes.

By contrast, when we switch to local direction laws in the colder regime
$\beta=O(1)$, this advantage largely disappears: both NURS and Barker exhibit
pronounced stickiness, reflecting the intrinsic difficulty of exploration under
highly localized proposals.

\begin{figure}[H]
  \centering
  \includegraphics[width=\linewidth]{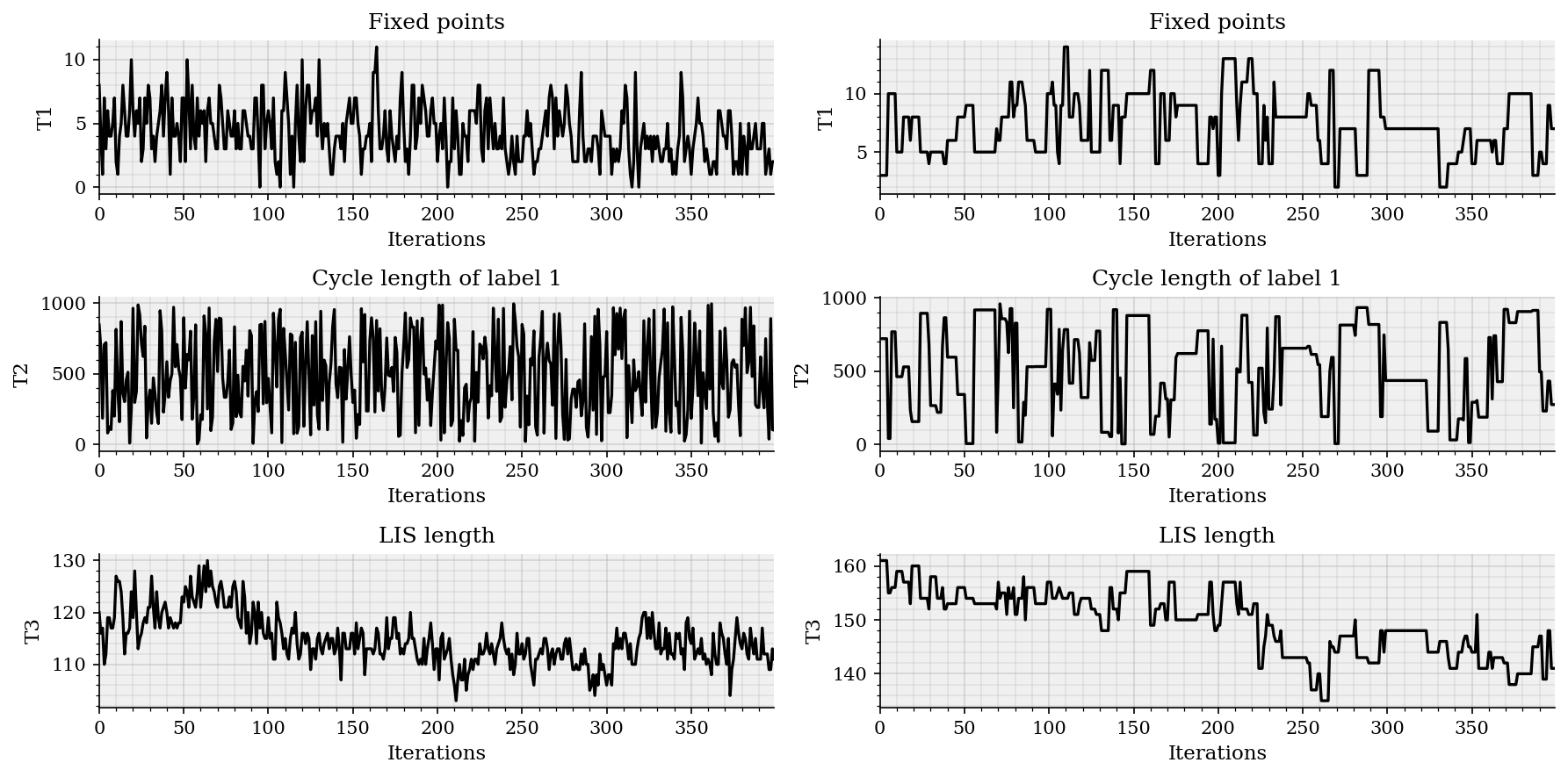}
  \caption{Trace diagnostics, Kendall with $\beta=n^{-1}$ and $B=9$. Left $M=7$ vs. right $M=1$ }
  \label{fig:trace-kendall}
\end{figure}
\begin{figure}[H]
  \centering
  \includegraphics[width=\linewidth]{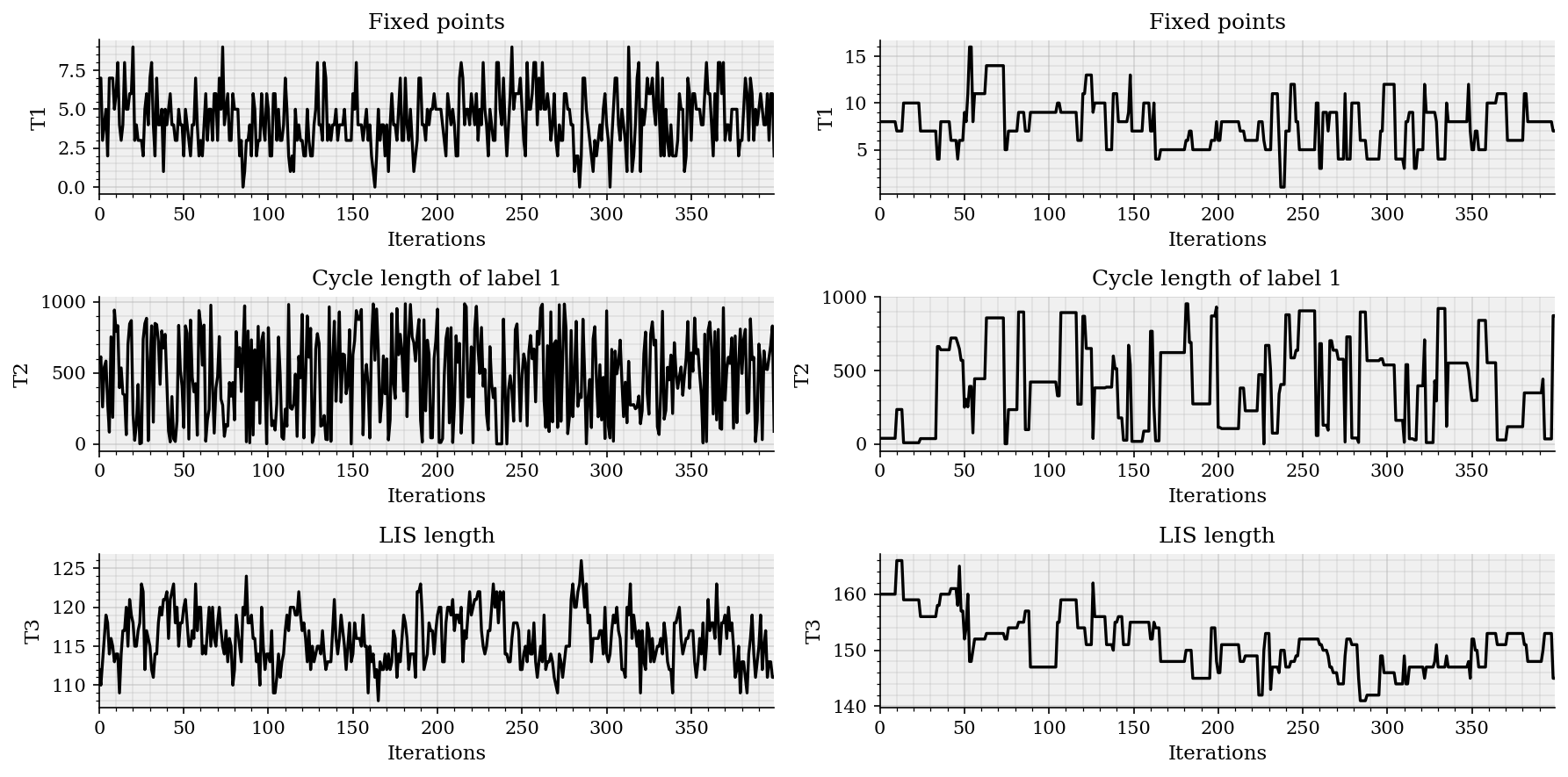}
  \caption{Trace diagnostics, $L^1$ with $\beta=n^{-1}$ and and $B=9$. Left $M=7$ vs. right $M=1$ }
  \label{fig:trace-L1}
\end{figure}
\begin{figure}[H]
  \centering
  \includegraphics[width=\linewidth]{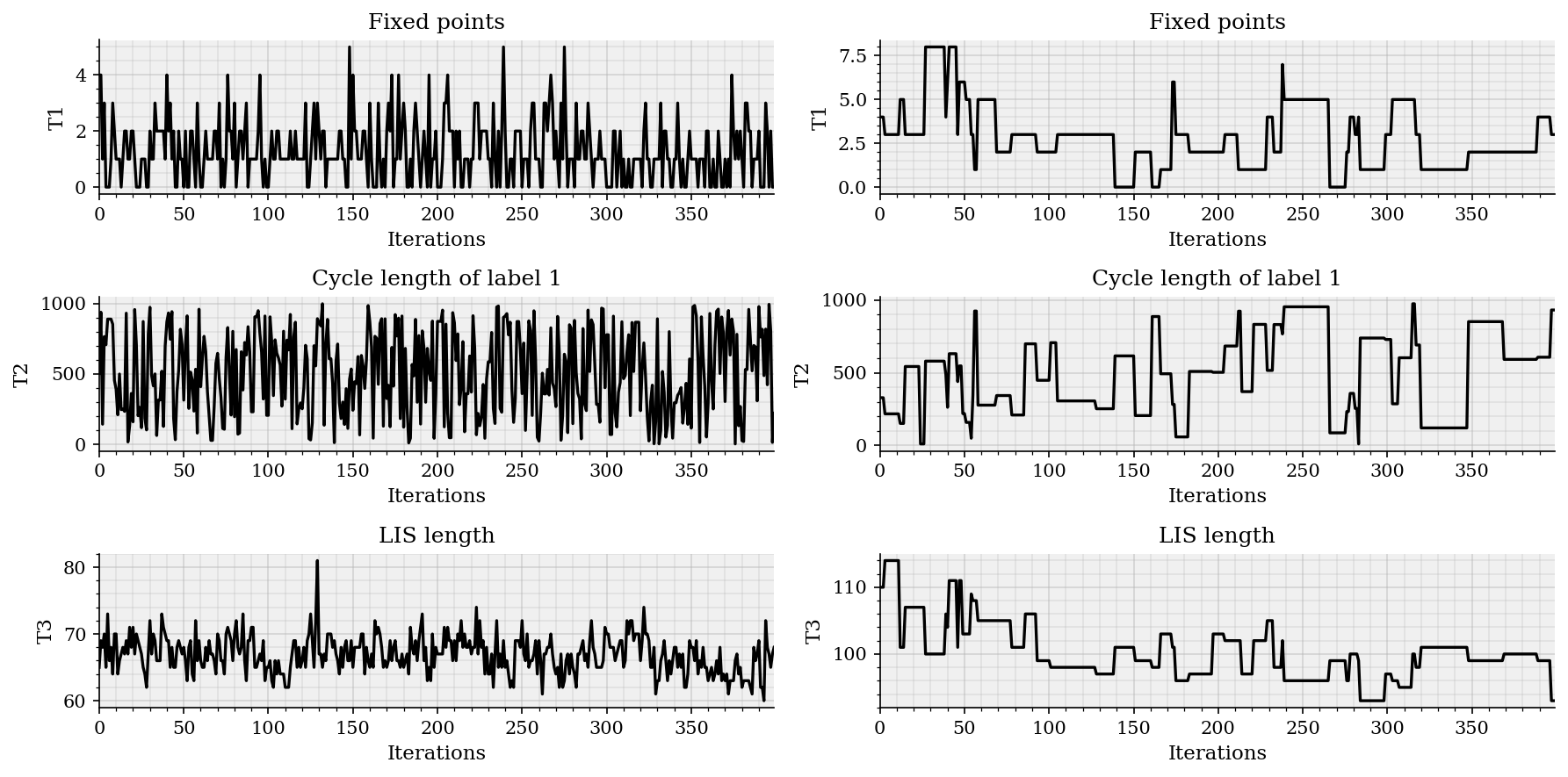}
  \caption{Trace diagnostics, $L^2$ with $\beta=n^{-2}$ and $B=15$. Left $M=7$ vs. right $M=1$.}
  \label{fig:trace-L2}
\end{figure}

\subsection{Discussion}

The numerical experiments illustrate how the
behavior of NURS is governed by the interplay between orbit geometry, energy
variation along orbits, and inverse temperature. In high--temperature regimes,
long orbits with mild energy variation produce approximately symmetric,
triangular index distributions,
effective state-space exploration, and rapid decorrelation, while in colder regimes the
locality of admissible moves becomes the dominant constraint. The observed
fixed--point statistics and trace diagnostics are consistent with known
asymptotic properties of the underlying Mallows models. 

\section{Beyond Permutations}

\label{sec:beyond}

The discrete No--Underrun Sampler (NURS) developed in this paper provides a locally adaptive Markov chain Monte Carlo method for sampling from
Mallows$(d,\sigma_0)$ models on the symmetric group~$S_n$.
Its construction rests on two structural ingredients:
(i) a family of invertible ``direction'' maps $\rho\in S_n$ generating
one--dimensional orbits
$\mathcal{O}=\{\sigma\circ\rho^k:k\in\mathbb{Z}\}$; and
(ii) an involution on an augmented space that re-centers these orbits while
preserving the joint law.
This involutive symmetry yields reversibility without a Metropolis correction.
Quantitative convergence is obtained, under additional assumptions, via an
orbitwise coupling argument that exploits overlap between truncated trajectories.

\medskip
\textbf{Toward general discrete spaces.}
A natural question is whether the same principles extend beyond permutations to
general discrete spaces~$\mathbb{S}$ equipped with a target measure~$\mu$.
At a formal level, this requires a family of bijective \emph{direction maps}
\[
T_v:\mathbb{S}\to\mathbb{S},
\qquad v\in\mathbb{V},
\]
indexed by an auxiliary direction space $\mathbb{V}$, such that the induced
orbits
\[
\mathcal{O}(\theta,v)=\{T_v^k(\theta):k\in\mathbb{Z}\}
\]
are well defined and sufficiently long.
Given such a family, one may perform randomized dyadic doubling along
$\mathcal{O}(\theta,v)$, terminate via a probabilistic no--underrun rule, and
resample the next state proportionally to the target weights.

On the augmented space $(\theta,v,m,b,i)$, the corresponding re-centering map
\[
\Psi(\theta,v,m,b,i)
\;=\;
\bigl(T_v^{\,i}(\theta),\,v,\,m,\,b-i,\,-i\bigr)
\]
is measure--preserving with respect to the product of counting measures.
As in the permutation setting, this involution guarantees reversibility in the
sense of Gibbs self--tuning (GIST).

At this level of abstraction, NURS may be viewed as a discrete analogue of Hamiltonian Monte
Carlo \cite{DuKePeRo1987,Ne2011} in a structural sense: both generate proposals
by following deterministic trajectories in an extended space and enforce
reversibility via an involutive symmetry.
In continuous Hamiltonian dynamics this symmetry relies on symplecticity and
momentum reversal, whereas in NURS it is imposed directly by the
measure--preserving involution $\Psi$ on the discrete augmented space.

\medskip
\textbf{Structural challenges.}
Despite this formal generality, several obstacles arise when $\mathbb{S}$ lacks
the group structure of the symmetric group $S_n$.
\begin{enumerate}
\item[(i)] \emph{Non-invertibility of local moves.}
If $T_v$ is not bijective, the re-centering map
$\Psi$ fails to preserve the product of counting
measures, invalidating the involution-based reversibility argument.
This occurs, for example, on a non-regular graph if $T_v(\theta)$ selects a
uniform random neighbor: the return probability then depends on vertex degree.   In such settings, some states have strictly more ways
to arrive than to leave, so augmented states cannot be paired one-to-one under
$\Psi$. Consequently, the involution cannot preserve counting measure, and
reversibility cannot be obtained without introducing explicit Metropolis-type
corrections.

\item[(ii)] \emph{Irregular energy variation along orbits.}
Quantitative convergence relies on a local comparability condition of the form
\[
e^{-\beta L_E}\;\le\;\frac{w_{t+m}}{w_t}\;\le\;e^{\beta L_E}
\quad\text{along each orbit}.
\]
If the energy $E$ varies sharply under $T_v$, then $\tanh(\beta L_E)$ approaches
one and the resulting contraction bound degenerates.
This phenomenon is the discrete analogue of the ``stiff Hamiltonian'' problem in
continuous HMC.

\end{enumerate}

Several structured classes of discrete state spaces nonetheless admit a natural NURS construction:
\begin{itemize}
\item \emph{Finite groups and semigroups:} right-multiplication by an invertible element yields well-defined orbits and involutions;
\item \emph{Regular graphs:} lifting the dynamics to oriented edges restores invertibility via unique reversals; and
\item \emph{Product spaces:} coordinatewise invertible moves (e.g.\ bit flips or local swaps) define factorized NURS directions.
\end{itemize}

\medskip
\textbf{Conjectures and open directions.}
The results above point toward a broader theory of \emph{orbit--adaptive reversible
MCMC}, in which reversibility and quantitative efficiency are governed by the
existence of measure--preserving orbit dynamics and by the geometry of energy
variation along those orbits. Natural open questions include when families of bijective,
orbit--equivariant directions are sufficient to construct nontrivial reversible
NURS--type kernels; when uniform overlap and controlled energy variation imply
logarithmic mixing times in structured spaces such as product spaces or groups;
and how orbit geometry and overlap parameters control spectral gaps. Further
questions concern the typical lengths of orbits generated by random directions,
the role of algebraic structure in ensuring long overlapping orbits, and the
design of mechanisms that reduce stiffness caused by large energy oscillations.
Taken together, these problems suggest that orbit--adaptive sampling lies
between local random walks and fully global sampling schemes.

\section*{Acknowledgements}

The authors are grateful to Bob Carpenter, Gilad Turok, Persi Diaconis, and Michael Howes for valuable discussions and insightful suggestions.  In particular, early
discussions with Persi Diaconis were instrumental in shaping the direction of
this work.

\printbibliography

\end{document}